\documentclass{amsart}
\usepackage{bbm}
\usepackage{mathrsfs}
\usepackage{amsfonts}
\usepackage{amsfonts,amsmath,amssymb,amsthm}
\usepackage{latexsym,amscd,graphics,color}
\usepackage{amsbsy}

\newtheorem{theorem}{Theorem}[section]
\newtheorem{lemma}[theorem]{Lemma}
\newtheorem{proposition}[theorem]{Proposition}
\newtheorem{corollary}[theorem]{Corollary}
\newtheorem{cliam}[theorem]{Claim}

\theoremstyle{definition}

\newtheorem{example}[theorem]{Example}

\theoremstyle{remark}
\newtheorem{remark}[theorem]{Remark}

\numberwithin{equation}{section}

\def\Hom{\mbox{Hom}}

\def\ad{\mbox{ad}}

\def\dim{\mbox{dim}}
\def\Aut{\mbox{Aut}}
\def\DAut{\mbox{DAut}}

\def\irr{\mbox{irr}}

\def\exp{\mbox{exp}}

\def\what#1{\widehat{#1}}

\begin{document}
\title{Generalized McKay Quivers, Root System and Kac-Moody Algebras}

\author{Bo Hou}
\address{College of Applied Science, Beijing University of
Technology, Beijing 100124, People's Republic China}
\email{houbo@emails.bjut.edu.cn}

\author{Shilin Yang}
\address{College of Applied Science, Beijing University of
Technology, Beijing 100124, People's Republic China}
\email{slyang@bjut.edu.cn}
\thanks{The second author was supported by the Science and
Technology Program of Beijing Education Committee (Grant
No. KM200710005013) and Foundation of Selected Excellent Science and
Technology Activity for Returned Scholars of Beijing.}




\begin{abstract}
Let $Q$ be a finite quiver and $G\subseteq\Aut(\mathbbm{k}Q)$ a
finite abelian group. Assume that $\widehat{Q}$ and $\Gamma$ is the
generalized Mckay quiver and the valued graph corresponding to $(Q,
G)$ respectively. In this paper we discuss the relationship between
indecomposable $\widehat{Q}$-representations and the root system of
Kac-Moody algebra $\mathfrak{g}(\Gamma)$. Moreover, we may lift $G$
to $\overline{G}\subseteq\Aut(\mathfrak{g}(\widehat{Q}))$ such that
$\mathfrak{g}(\Gamma)$ embeds into the fixed point algebra
$\mathfrak{g}(\widehat{Q})^{\overline{G}}$ and
$\mathfrak{g}(\widehat{Q})^{\overline{G}}$ as
$\mathfrak{g}(\Gamma)$-module is integrable.
\end{abstract}

\maketitle
\section{Introduction} \label{sect-1}
Thirty years ago, McKay introduced a class of quivers, now called
the McKay quivers, for some finite subgroups of the general linear
group \cite{Mc}. Let $\mathbb{C}$ denote the complex number field.
McKay observed that the McKay quiver for $G\subseteq
\mathrm{SL}(2,\mathbb{C})$ is the double quiver of the extended
Dynkin quiver $\widetilde{A}_{n}$, $\widetilde{D}_{n}$,
$\widetilde{E}_{6}$, $\widetilde{E}_{7}$, $\widetilde{E}_{8}$
respectively. Furthermore, the corresponding Dynkin diagram is the
same as the one occurring in the minimal resolution of singularities
for the quotient surface $\mathbb{C}/G$ (see \cite{Br}). McKay
quiver has played an important role in many mathematical fields such
as quantum group, algebraic geometry, mathematics physics and
representation theory (see, for examples \cite{Au, CH, GM, G, Lu,
Re}).

Let $V$ be a finite vector space over a field $\mathbbm{k}$ of
characteristic 0 and $G\subseteq \mathrm{GL}_{\mathbbm{k}}(V)$  a
finite group. Assume that $\mathrm{T}_{\mathbbm{k}}(V)$ is  the
tensor algeba of $V$ over $\mathbbm{k}$. It is well-known that the
skew group algebra $\mathrm{T}_{\mathbbm{k}}(V)\ast G$ is Morita
equivalent to the path algebra $\mathbbm{k}\what{Q}$, where
$\what{Q}$ is the McKay quiver of $G$ (see \cite{GM}). In other
words, the McKay quiver  realizes the Gabriel quiver of
$\mathrm{T}_{\mathbbm{k}}(V)\ast G$. It is natural to ask how to
determine the Gabriel quiver of skew group algebra $\Lambda\ast G$
for any algebra $\Lambda$. Recently, for any path algebra
$\mathbbm{k}Q$ over an algebraically closed field $\mathbbm{k}$ and a
finite group $G$ such that char$\mathbbm{k}\nmid |G|$, if the action
of $G$ on $\mathbbm{k}Q$ permutes the set of primitive idempotents
and stabilizing the vector space spanned by the arrows, Demonet in
\cite{De} has constructed a quiver $\widehat{Q}$ such that the path
algebra $\mathbbm{k}\widehat{Q}$ is Morita equivalent to the skew
group algebra $\mathbbm{k}Q\ast G$. The quiver $\widehat{Q}$ can be
viewed as a generalization of McKay quiver, which is called the
generalized McKay quiver of $(Q, G)$ in this paper.

Given a finite quiver $Q$ with an admissible automorphism ${\bf a}$.
Hubery in \cite{Hu, Hu1} described the correspondence between
dimension vectors of the isomorphically invariant
$Q$-indecomposables and the positive root system of ${\frak
g}(\Gamma)$, where $\Gamma$ is the valued graph of $(Q, {\bf a})$.
Motivated by Hubery's work, the aim of this paper is to establish the
correspondence between the indecomposable
$\widehat{Q}$-representations and the positive roots of the
symmetrizable Kac-Moody algebra $\mathfrak{g}(\Gamma)$ of the valued
graph $\Gamma$ associated to $(Q, G)$, where $Q$ is a finite quiver
and $G$ is a finite abelian  automorphism group of $\mathbbm{k}Q$.
Moreover, we can lift $G$ to an automorphism group $\overline{G}$ of
Kac-Moody algebra $\mathfrak{g}:=\mathfrak{g}(\widehat{Q})$ of
$\widehat{Q}$, such that $\mathfrak{g}(\Gamma)$ can be embedded into
the fixed point subalgebra $\mathfrak{g}^{\overline{G}}$. In this
case, we also show that $\mathfrak{g}^{\overline{G}}$ as a
$\mathfrak{g}(\Gamma)$-module is integrable. Compared with Hubery's
work, a more general  description is given by approach of the
generalized McKay quiver.

\medskip
For a finite quiver $Q=(I, E)$ and a finite abelian group
$G\subseteq\Aut(\mathbbm{k}Q)$ (the algebra automorphism group of
$\mathbbm{k}Q$). We always assume that the action of $G$ on $Q$ is
admissible, i.e., no arrow connects to vertices in the same orbit.
Then we can get a valued graph $\Gamma$ without loops and a
generalized McKay quiver $\widehat{Q}$ corresponding to $(Q, G)$. By
\cite{RR}, we can define an action of $G$ on
$\mathbbm{k}\widehat{Q}$ due to the Morita equivalence between the skew
group algebra $\mathbbm{k}Q\ast G$ and $\mathbbm{k}\widehat{Q}$.
Therefore this action induces an action on
$\widehat{Q}$-representations. Let $G_{X}$ be a complete set of left
coset representatives of $H_{X}=\{g\in G \mid {^{g}X}\cong X\}$ in
$G$ for any $\widehat{Q}$-representation $X$, let $\mathbb{Z}I$,
$\mathbb{Z}\widehat{I}$ and $\mathbb{Z}\mathcal{I}$ be the root
lattice of $Q$, $\widehat{Q}$ and $\Gamma$, respectively. Applying
the equivalence between representation category of $\widehat{Q}$ and
module category of the skew group algebra $\mathbbm{k}Q\ast G$ and
the fact that each $\mathbbm{k}Q\ast G$ module as a
$Q$-representation is $G$-invariant, we  define a map
$$
h:\quad\mathbb{Z}\widehat{I}\longrightarrow(\mathbb{Z}I)^{G}
\longrightarrow\mathbb{Z}\mathcal{I}
$$
where $(\mathbb{Z}I)^{G}$ is the fixed point set of $\mathbb{Z}I$
under the action of $G$.  The map $h$ builds a bridge between the
dimension vectors of indecomposable $\widehat{Q}$-representations
and the root system of Kac-Moody algebra $\mathfrak{g}(\Gamma)$. The
first main result of this paper is described as follows.

\begin{theorem}\label{thm1-1}
Let $Q$ be a quiver without loops and with an admissible action of a
finite abelian subgroup $G\subseteq{\rm Aut}(\mathbbm{k}Q)$, where
$\mathbbm{k}$ is an algebraically closed field with ${\rm
char}\mathbbm{k}\nmid |G|$. Assume that $\Gamma$ and $\widehat{Q}$
is the valued graph and generalized McKay quiver associated to $(Q,
G)$. Then

(1) the images under $h$ of the dimension vectors of all the
indecomposable $\widehat{Q}$-representations give the positive root
system of the symmetrisable Kac-Moody algebra
$\mathfrak{g}(\Gamma)$;

(2) for each positive real root $\alpha$ of $\mathfrak{g}(\Gamma)$,
let $X$ be a $\widehat{Q}$-representation such that $h({\bf
dim}X)=\alpha$. Then there are $|G_{X}|$ indecomposable
$\widehat{Q}$-representations (up to isomorphism) such that their
dimension vectors under $h$ are $\alpha$.
\end{theorem}

The proof of this theorem is based on understanding the relationship
among indecomposable $\widehat{Q}$-representations, indecomposable
$\mathbbm{k}Q\ast G$-modules and indecomposable $G$-invariant
$Q$-representations. In the proof, we also need the dual between
$(Q, G)$ and $(\widehat{Q}, G)$. This duality is first discussed in
\cite{RR} for a finite quiver with an automorphism. Here we give a
general and strict proof by the generalized McKay quiver.

\medskip
Next we consider the relationship between Kac-Moody algebra
$\mathfrak{g}(\Gamma)$ and the fixed point subalgebra
$\mathfrak{g}^{\overline{G}}$. The action of $G$ on $\widehat{Q}$
naturally induces an action on the derived algebra $\mathfrak{g}'$
of $\mathfrak{g}$. Let $\Omega=\{g_{1}, g_{2},\cdots,g_{n}\}$ be a
set of generators of $G$. Following from \cite{KW}, we lift $G$ to
$\overline{G}\subseteq\Aut(\mathfrak{g})$ corresponding to a family
of linear maps $\{\psi_{i}:=\psi_{g_{i}} :
\mathfrak{h}/\mathfrak{h}'\rightarrow\mathfrak{c}\mid
g_{i}\in\Omega\}$, where $\mathfrak{c}$ is the center of
$\mathfrak{g}$, $\mathfrak{h}$ and $\mathfrak{h}'$ is the Cartan
subalgebra of $\mathfrak{g}$ and $\mathfrak{g}'$ respectively.
Denote by $C$ the symmetrisable generalized Cartan matrix of the
valued graph $\Gamma$. Then, we can give a realization
$(\mathcal{H}^{\overline{G}}, \{\epsilon_{i}\}, \{h_{i}\})$ of $C$
by the fixed point set $\mathfrak{h}^{\overline{G}}$ of
$\mathfrak{h}$, and we obtain that

\begin{theorem}\label{thm1-2}
For the lifting $\overline{G}$ of $G$ corresponding to $\{\psi_{i} :
\mathfrak{h}/\mathfrak{h}'\rightarrow\mathfrak{c}\mid
g_{i}\in\Omega\}$ such that
$\psi_{i}\big((\mathcal{H}+\mathfrak{h}')/\mathfrak{h}'\big)=0$,
there is a monomorphism
$$\mathfrak{g}(\Gamma)\rightarrow\mathfrak{g}^{\overline{G}}.$$
Moreover this monomorphism endows $\mathfrak{g}^{\overline{G}}$ with
an integrable $\mathfrak{g}(\Gamma)$-module structure under the
adjoint action of $\mathfrak{g}(\Gamma)$. In particular, if $Q$ is a
finite union of Dynkin quivers, then
$\mathfrak{g}(\Gamma)\cong\mathfrak{g}^{\overline{G}}$ as Lie
algebras.
\end{theorem}

In the end of this paper, two examples are given to elucidate our
results.

\medskip
Throughout this paper, let $\mathbbm{k}$ denote an algebraic closed
field and $\mathbb{Z}$ denote the set of integers. We denote by $G$
the finite group such that char$\mathbbm{k}\nmid |G|$, denote by
${\bf mod}$-$\Lambda$ the category of (right) $\Lambda$-modules for
any $\mathbbm{k}$-algebra $\Lambda$.


\section{Preliminaries} \label{sect-2}

\noindent {\bf 2.1.} Recall that a quiver $Q=(I, E)$ is an oriented
graph with $I$ the set of vertices and $E$ the set of arrows. A
quiver $Q$ is said to be finite if $I$ and $E$ are all finite set.
An arrow in $Q$ is called a loop if its staring vertex coincides
with its terminating vertex. In this paper we only consider a finite
quiver without loops. Therefore we have a path algebra
$\mathbbm{k}Q$  for a quiver $Q$ (see \cite{ASS, ARS}).

A representation $X=(X_{i}, X_{\alpha})$ of the quiver $Q=(I, E)$
consists of a family of $\mathbbm{k}$-vector spaces $X_{i}$ for
$i\in I$, together with a family of $\mathbbm{k}$-linear maps
$X_{\alpha}: X_{i}\rightarrow X_{j}$ for $\alpha: i\rightarrow j$ in
$E$. Given two representations $X$ and $Y$ of $Q$, a morphism
$\varphi: X\rightarrow Y$  is given by a family of
$\mathbbm{k}$-linear maps $\varphi_{i}: X_{i}\rightarrow Y_{i}~(i\in
I)$ such that $\varphi_{j}\circ
X_{\alpha}=Y_{\alpha}\circ\varphi_{i}$ for each arrow $\alpha:
i\rightarrow j$. It is well-known that the category of
representations of $Q$ is naturally equivalent to the category of
$\mathbbm{k}Q$-modules (see \cite{ASS, ARS}). Thus we always
identify a $\mathbbm{k}Q$-module $X$ with a $Q$-representation
$(X_{i}, X_{\alpha})$ in this paper.

\medskip
\noindent {\bf 2.2.} Assume that $\Lambda$ is a
$\mathbbm{k}$-algebra and $G$ acts on $\Lambda$, the skew group
algebra of $\Lambda$ under the action of $G$ is by definition the
$\mathbbm{k}$-algebra whose underlying $\mathbbm{k}$-vector space is
$\Lambda\otimes_{\mathbbm{k}}\mathbbm{k}[G]$ and whose
multiplication is defined by
$$
(\lambda\otimes g)(\lambda'\otimes g')=\lambda g(\lambda')\otimes gg'
$$
for all $\lambda, \lambda'\in \Lambda$ and $g, g'\in G$ (see
\cite{RR}). For convenience, we denote this algebra by $\Lambda\ast
G$, denote the element $\lambda\otimes g$ in $\Lambda\ast G$ by
$\lambda g$.  Note that $\Lambda$ and $\mathbbm{k}[G]$ can be viewed
as subalgebras of $\Lambda\ast G$.

Let $\Lambda=\mathbbm{k}Q$ be the path algebra for a quiver $Q=(I,
E)$. Assume that $G$ acts on $\mathbbm{k}Q$ permuting the set of
primitive idempotents $\{e_{i}\mid i\in I\}$ and stabilizing the
vector space spanned by the arrows. Let $\mathcal {I}$ denote a set
of representatives of class of $I$ under the action of $G$. For any
$i\in I$, let $G_{i}$ denote the subgroup of $G$ stabilizing
$e_{i}$, For each $i\in I$, there exist some $g\in G$ such that
$g^{-1}(i)\in\mathcal{I}$. We fix such a $g$ and denote it by
$\kappa_{i}$. Let $\mathcal {O}_{i}$ be the orbit of $i$ under the
action of $G$. For $(i, j)\in\mathcal {I}^{2}$, $G$ acts on
$\mathcal {O}_{i}\times\mathcal {O}_{j}$ by the diagonal action. A set
of representatives of the classes of this action will be denoted by
${\mathcal F}_{ij}$.

For $i, j\in I$,  we denote by $E_{ij}\subseteq\mathbbm{k}Q$ the
vector space spanned by the arrows from $i$ to $j$ and regard it as
a left and right $\mathbbm{k}[G_{ij}]:=\mathbbm{k}[G_{i}\cap
G_{j}]$-module by restricting the action of $G$. In \cite{De}
Demonet defined the quiver $\widehat{Q}=(\widehat{I}, \widehat{E})$
as follows $$
\widehat{I}=\bigcup_{i\in\mathcal {I}}\{i\}\times\mbox{irr}G_{i},
$$
where $\mbox{irr}G_{i}$ is a set of representatives of isomorphism
classes of irreducible representations of $G_{i}$. The set of arrows
of $\widehat{Q}$ from $(i, \rho)$ to $(j, \sigma)$ is a basis of
$$
\bigoplus_{(i',j')\in {\mathcal F}_{ij}}\Hom_{\mathbbm{k}[G_{i'j'}]}
\left((\rho\cdot\kappa_{i'})|_{G_{i'j'}},
~(\sigma\cdot\kappa_{j'})|_{G_{i'j'}}\otimes_{\mathbbm{k}}E_{i'j'}\right),
$$
where the representation $\rho\cdot\kappa_{i'}$ of $G_{i'}$ is the
same as $\rho$ as a $\mathbbm{k}$-vector space, and
$(\rho\cdot\kappa_{i'})g=\rho\kappa_{i'}g\kappa_{i'}^{-1}$ for each
$g\in G_{i'}=\kappa_{i'}^{-1}G_{i}\kappa_{i'}$. Demonet yielded the
following theorem.

\begin{theorem}\label{thm2-1}
{\rm (see \cite{De})} The category {\bf
mod}-$\mathbbm{k}\widehat{Q}$ is equivalent to the category {\bf
mod}-$\mathbbm{k}Q\ast G$.
\end{theorem}

In particular, if the quiver $Q$ is a singular vertex with $m$
loops, we can view $G$ as a subgroup of ${\bf GL}_{m}(\mathbbm{k})$.
Then the quiver $\widehat{Q}$ is just the McKay quiver of $G$. Thus,
we view $\widehat{Q}$ as a generalization of McKay quiver in
general. Furthormore, for any factor algebra $\mathbbm{k}Q/J$, the
shew group algebra $(\mathbbm{k}Q/J)\ast G$ is Morita equivalent to
a factor algebra of $\mathbbm{k}\widehat{Q}$. This implies that the
generalized McKay quiver can realize the Garbiel quiver of
$\Lambda\ast G$ for any basic algebra $\Lambda$.

\medskip
\noindent {\bf 2.3.}  For a quiver $Q=(I, E)$, there is a
corresponding symmetric generalized Cartan matrix $A=(a_{ij})$
indexed by $I$ with entries
$$
a_{ij}=\left\{ \begin{array}{ll}
2,& i=j;\\
-|\{\mbox{edges between vertices } i \mbox{ and } j\}|,& i\neq j.
\end{array}\right.
$$
It is obvious that $A$ is independent of the orientation of $Q$.

Denote by $\mathfrak{g}(Q)$ for the associated symmetric Kac-Moody
algebra corresponding to $A$ with the simple root set
$\Pi=\{\varepsilon_{i}\mid i\in I\}$ and root system $\Delta_{Q}$.
The root lattice $\mathbb{Z}I$ of $Q$ is the free abelian group on
$\Pi$, with the partially order such that $\alpha=\sum_{i\in
I}\alpha_{i}\varepsilon_{i}\ge 0$ if and only if $\alpha_{i}\geq0$
for all $i\in I$. We endow $\mathbb{Z}I$ with a symmetric bilinear
form $(-,-)_{Q}$ via $(\varepsilon_{i},
\varepsilon_{j})_{Q}=a_{ij}$. Then, for each vertex $i\in I$, we
have a reflection $r_{i} :
\alpha\mapsto\alpha-(\alpha,\varepsilon_{i})_{Q}\varepsilon_{i}$.
These reflections generate the Weyl group $\mathcal{W}(Q)$ of $Q$.
The real roots of $Q$ are given by the images under $\mathcal
{W}(Q)$ of the simple roots $\varepsilon_{i}$ and the imaginary
roots are given by $\pm$ the images under $\mathcal {W}(Q)$ of the
fundamental set
$$F_{Q}:=\{\alpha> 0\mid (\alpha,
\varepsilon_{i})_{Q}\leq0 \mbox{ for all } i \mbox{ and the support
 of } \alpha \mbox{ is connected}\}.$$

Suppose that the action of $G$ on path algebra $\mathbbm{k}Q$
permutes the set of primitive idempotents. The action of $G$ is
said to be admissible if no arrow connects to vertices in the same
$G$-orbit. For any quiver $Q$ with an admissible action of $G$, we
can construct a symmetric matrix $B=(b_{ij})$ indexed by $\mathcal
{I}$,  where
$$
b_{ij}=\left\{ \begin{array}{ll}
2|\mathcal{O}_{i}|,& i=j;\\
-|\{\mbox{edges between vertices in } \mathcal{O}_{i} \mbox{ and }
\mathcal{O}_{j}\}|,& i\neq j.
\end{array}\right.
$$
Let $d_{i}:=b_{ii}/2=|\mathcal{O}_{i}|$ and $D=\mbox{diag}(d_{i})$.
Then $C=(c_{ij})=D^{-1}B$ is a symmetrisable generalized Cartan
matrix indexed by $\mathcal {I}$. It is well-known that there is a
unique valued graph $\Gamma$ corresponding to the matrix $C$ by
\cite{DR}. The valued graph $\Gamma$ has the vertex set $\mathcal
{I}$ and an edge $i$----$j$ equipped with the ordered pair
$(|c_{ji}|, |c_{ij}|)$ whenever $c_{ij}\neq0$. Since the action of
$G$ is admissible, $\Gamma$ has no loops. For each connected
component $\Gamma'$ of the graph $\Gamma$, we always take the
representative set $\mathcal{I}$ such that the underlying graph of
the full subquiver generated by the vertices in $\Gamma'$ is
connected.

Denote by $\mathfrak{g}(\Gamma)$ for the associated symmetric
Kac-Moody algebra corresponding to $C$. The simple root set and root
system of $\Gamma$ are denoted by
$\Pi_\Gamma=\{\overline{\varepsilon}_{i}\mid i\in \mathcal{I}\}$ and
$\Delta_{\Gamma}$. Let $\mathbb{Z}\mathcal{I}$ denote the root
lattice of $\Gamma$. There is a symmetric
bilinear form $(-,-)_{\Gamma}$ determined by $B$ on
$\mathbb{Z}\mathcal{I}$ such that $(\overline{\varepsilon}_{i},
\overline{\varepsilon}_{j})_{\Gamma}=b_{ij}$, and a reflection
$\gamma_{i}$ on $\mathbb{Z}\mathcal{I}$ defined by
$$\gamma_{i} :
\alpha\mapsto\alpha-\frac{1}{d_{i}}(\alpha,
\overline{\varepsilon}_{i})_{\Gamma}\overline{\varepsilon}_{i}$$
for each $i\in{\mathcal I}$. These reflections generate the Weyl group
$\mathcal {W}(\Gamma)$ of $\Gamma$. Similarly, we have the real
roots and the imaginary roots associated to $\Gamma$ (see
\cite{Ka}).


\section{ Proof of Theorem \ref{thm1-1}} \label{sect-3}
From now on, unless otherwise stated we fix a finite group
$G\subseteq\Aut(\mathbbm{k}Q)$ and assume that the action of $G$ is
admissible. Let $\widehat{Q}$ and $\Gamma$ be the generalized Mckay
quiver and the valued graph corresponding to $(Q, G)$. In this
section, we show that the correspondence between indecomposable
representations of $\widehat{Q}$ and the positive root system of
$\Gamma$.

\medskip
\noindent {\bf 3.1.} The group $G$ acts naturally on the root
lattice $\mathbb{Z}I$, i.e., $g(\varepsilon_{i})=\varepsilon_{g(i)}$
for any $g\in G$. It is easy to check that this action preserves the
partial order $\geq$ and the bilinear form $(-,-)_{Q}$ is
$G$-invariant. Let
$$
(\mathbb{Z}I)^{G}:=\{\alpha\in\mathbb{Z}I\mid g(\alpha)=\alpha
\mbox{ for any }g\in G\}.
$$
There is a canonical bijection
$$
f :\quad
(\mathbb{Z}I)^{G}\longrightarrow\mathbb{Z}\mathcal{I}
$$
given by
$$
f\Big(\sum_{i\in I}\alpha_{i}\varepsilon_{i}\Big)=\sum_{i\in
\mathcal {I}}\alpha_{i}\overline{\varepsilon}_{i}.
$$
The admissibility of the
action of $G$ implies that the reflections $r_{i}$ and $r_{j}$
commute whenever $i$ and $j$ lie in the same $G$-orbit. Therefore
the element
$$
S_{i} :=\prod_{i'\in\mathcal {O}_{i}} r_{i'}\in
\mathcal {W}(Q)
$$
is well-defined for any $i\in\mathcal {I}$. Note that $g\circ r_{i}=
r_{g(i)}\circ g$ for any $g\in G$, we have $S_{i}\in C_{G}(\mathcal
{W}(Q))$, the set of elements in the Weyl group commuting with the
action of $G$. By induction on the length of the element in
$C_{G}(\mathcal {W}(Q))$, it is easy to check that $C_{G}(\mathcal
{W}(Q))$ is generated by $S_{i}$, $i\in\mathcal {I}$.

\medskip
Similar to \cite[Lemma 3]{Hu1}, we have

\begin{lemma}\label{lem3-1}
For any $\alpha, \beta\in (\mathbb{Z}I)^{G}$,
we have

(1) $(\alpha, \beta)_{Q}=(f(\alpha), f(\beta))_{\Gamma}$;

(2) $f(S_{i}(\alpha))=\gamma_{i}(f(\alpha))\in
\mathbb{Z}\mathcal{I}$ for $i\in \mathcal {I}$.

(3) The map $\gamma_{i}\mapsto S_{i}$ induces a group isomorphism
$\mathcal{W}(\Gamma)\stackrel{\simeq}{\longrightarrow}C_{G}(\mathcal
{W}(Q))$.
\end{lemma}

\begin{proof}  (1) Set
$\varepsilon^{i}:=\sum_{i'\in\mathcal{O}_{i}}\varepsilon_{i'}$. Then
$\{\varepsilon^{i}\mid i\in \mathcal{I}\}$ is a basis of
$(\mathbb{Z}I)^{G}$. Since
$$
(\varepsilon^{i},
\varepsilon^{j})_{Q}=\sum_{i'\in\mathcal{O}_{i}\atop
j'\in\mathcal{O}_{j}}a_{i'j'}=b_{ij}=(\overline{\varepsilon}_{i},
\overline{\varepsilon}_{j})_{\Gamma}
$$
for any $i, j\in\mathcal{I}$, (1) is obvious.

(2) Since the bilinear form $(-,-)_{Q}$ is $G$-invariant, we have
$$
S_{i}(\alpha)=\alpha-\sum_{i'\in\mathcal{O}_{i}}(\alpha,
\varepsilon_{i'})_{Q}\varepsilon_{i'}=\alpha-\sum_{i'\in\mathcal{O}_{i}}
\frac{1}{d_{i}}\big(\alpha,
\sum_{j\in\mathcal{O}_{i}}\varepsilon_{j}\big)_{Q}\varepsilon_{i'}
=\alpha-\frac{1}{d_{i}}(f(\alpha),
\overline{\varepsilon}_{i'})_{\Gamma}\varepsilon^{i}
$$
by (1). We obtain that
$$
f(S_{i}(\alpha))=f(\alpha)-\frac{1}{d_{i}}(f(\alpha),
\overline{\varepsilon}_{i'})_{\Gamma}\overline{\varepsilon}_{i}
=\gamma_{i}(f(\alpha)).
$$

(3) By the result of (2), it is easy to check that $\gamma_{i}$ and
$S_{i}$ satisfy the same relations. Thus
$\mathcal{W}(\Gamma)\cong C_{G}(\mathcal {W}(Q))$.
\end{proof}

For a given $\alpha\in \mathbb{Z}I$, let $H_{\alpha}=\{g\in G \mid
g(\alpha)=\alpha\}$. Then $H_{\alpha}$ is a subgroup of $G$. We
denote by $G_{\alpha}$ a complete set of left coset representatives
of $H_{\alpha}$ in $G$, and let
$$
\Sigma(\alpha):=\sum_{g\in G_{\alpha}}g(\alpha).
$$
Obviously, $\Sigma(\alpha)\in(\mathbb{Z}I)^{G}$ and we have

\begin{lemma}\label{lem3-2}
The map $\alpha\mapsto f(\sigma(\alpha))$ induces a surjection
$\pi : \Delta_{Q}\rightarrow \Delta_{\Gamma}$. Moreover, if
$f(\sigma(\alpha))$ is a real root, $\alpha$ has to be real and
unique up to $G$-orbit.
\end{lemma}

\begin{proof} First, for any $\omega\in C_{G}(\mathcal {W}(Q))$, we have
$H_{\alpha}=H_{\omega(\alpha)}$ since the action of $C_{G}(\mathcal
{W}(Q))$ and the action of $G$ on $\mathbb{Z}I$ is commutative. Thus
we can take $G_{\alpha}=G_{\omega(\alpha)}$ for any $\omega\in
C_{G}(\mathcal {W}(Q))$.

We now consider $\beta:=\omega'(f(\Sigma(\alpha)))$ with
$\omega'\in\mathcal{W}(\Gamma)$. Let $\omega\in C_{G}(\mathcal
{W}(Q))$ be the element corresponding to $\omega'$ under the
isomorphism in Lemma \ref{lem3-1}(3). Then
$\beta=f(\omega(\Sigma(\alpha)))=f(\Sigma(\omega(\alpha)))$ has
connected support since the support of $\alpha$ is connected. It is
either positive or negative since $\Sigma$ preserves the partial
order $\geq$. Denote by $\mathcal{O}_{\beta}$ the orbit of $\beta$
under the action of $\mathcal{W}(\Gamma)$. Then

\begin{itemize}
\item if all elements in $\mathcal{O}_{\beta}$ are positive,
the element in $\mathcal{O}_{\beta}$ with minimal height lies in
$F_{\Gamma}$;
\item if all elements in $\mathcal{O}_{\beta}$ are negative,
the element in $\mathcal{O}_{\beta}$ with maximal height lies in
$-F_{\Gamma}$;
\item otherwise, there exists a positive number $m$ and
$i\in\mathcal{I}$ such that
$m\overline{\varepsilon}_{i}\in\mathcal{O}_{\beta}$.
\end{itemize}
In the last case, we have $\omega(\alpha)=m\varepsilon_{i'}$ for
some $\omega\in\mathcal{W}(Q)$, $i'\in\mathcal {O}_{i}$. But
$\omega(\alpha)\in\Delta_{Q}$, we must have $m=1$ and so that
$\overline{\varepsilon}_{i}\in\mathcal{O}_{\beta}$. Thus $\beta$ is
a root of $\Gamma$ and $\pi : \Delta_{Q}\rightarrow\Delta_{\Gamma}$,
$\alpha\mapsto f(\Sigma(\alpha))$ is well-defined.

Now, we prove that the map $\pi$ is surjective. Here we only need to
show that $F_{\Gamma}$ lies in the image of $\pi$. For any $\beta\in
F_{\Gamma}$, $\gamma:=f^{-1}(\beta)$ satisfies
$$
0\geq(\beta, \overline{\varepsilon}_{i})_{\Gamma}=(\gamma,
\Sigma(\varepsilon_{i'}))_{Q}=\sum_{g\in
G_{\varepsilon_{i'}}}(\gamma, g(\varepsilon_{i'}))_{Q}=d_{i}(\gamma,
\varepsilon_{i'})_{Q}
$$
for any $i\in\mathcal{I}$ and $i'\in\mathcal{O}_{i}$. Thus any
connected component $\alpha$ of $\gamma$ lies in $F_{Q}$ and
$\Sigma(\alpha)=\gamma$. By Lemma \ref{lem3-1}(3) we get the proof.
\end{proof}

For any $g\in G$, we have an additive autoequivalence functor
$$
\begin{aligned}F_{g}:\quad\mbox{{\bf
mod}-}\mathbbm{k}Q\quad &\longrightarrow \quad\mbox{{\bf
mod}-}\mathbbm{k}Q \\
M  \qquad&\mapsto  \qquad{^{g}M}
\end{aligned}
$$
where the $\mathbbm{k}Q$-module ${^{g}M}$ is defined by taking the
same underlying vector space as $M$ with the action
$m\cdot\lambda=mg(\lambda)$ for $m\in M$ and
$\lambda\in\mathbbm{k}Q$, and $F_{g}(\psi)=\psi$ for any
homomorphism $\psi : M\rightarrow N$. Let $(M_{i},
~M_{\alpha})_{i\in I, \alpha\in E}$ be the $Q$-representation
corresponding to $M$. Then the $Q$-representation $^{g}M$ is
$(^{g}X_{i}, {^{g}X}_{\alpha})_{i\in I, \alpha\in E}$, where
$^{g}X_{i}=X_{g(i)}$ and $^{g}X_{\alpha}=\sum_{\beta}\zeta_{\beta}
X_{\beta}$ if $g(\alpha)=\sum_{\beta}\zeta_{\beta}\beta$, $\beta\in
E$, $\zeta_{\beta}\in\mathbbm{k}$.

A $\mathbbm{k}Q$-module $M$ is said to be $G$-invariant if
$F_{g}(M)\cong M$ for any $g\in G$, a $G$-invariant
$\mathbbm{k}Q$-module $M$ is said to be indecomposable $G$-invariant
if $M$ is non-zero and $M$ cannot be written as a direct sum of two
non-zero $G$-invariant $\mathbbm{k}Q$-modules. It is known that
$\mathbbm{k}Q$-module $M$ has a $\mathbbm{k}Q\ast G$-module
structure if and only if $M$ is $G$-invariant, and the full
subcategory of $\mbox{{\bf mod}-}\mathbbm{k}Q$ generated by the
$G$-invariant $\mathbbm{k}Q$-module is also a Krull-Schmidt category
(see \cite{HY}).

\medskip
For a given $\mathbbm{k}Q$-module $M$, we let $H_{M}:=\{g\in G\mid
F_{g}(M)\cong M\}$ and $G_{M}$ be a complete set of left coset
representatives of $H_{M}$ in $G$. Then for each
$\mathbbm{k}Q$-module $M$, we  define a $G$-invariant
$\mathbbm{k}Q$-module
$$
\sum(M):=\bigoplus_{g\in G_{M}}{^{g}M}.
$$
It is easy to see that each $G$-invariant $\mathbbm{k}Q$-module has
this form. For each $\mathbbm{k}Q$-module $M$, we denote the
dimension vector of $M$ by the linear combination ${\bf dim}
X:=\sum_{i\in I}\dim X_{i}\,\varepsilon_{i}\in\mathbb{Z}I$. It
is easy to see that ${\bf dim}F_{g}(M)=g({\bf dim}M)$ for any $g\in
G$ and $M\in\mbox{{\bf mod}-}\mathbbm{k}Q$.

\begin{proposition} \label{prop3-3}
For any indecomposable $G$-invariant $\mathbbm{k}Q$-module $M$,
$f({\bf dim}M)$ is a root of $\Gamma$. Moreover, for any positive
real root $\beta$ of $\Gamma$, there is a unique {\rm(}up to
isomorphism{\rm )} indecomposable $G$-invariant
$\mathbbm{k}Q$-module $M$ with $\frac{1}{2}({\bf dim}M, {\bf
dim}M)_{Q}$ indecomposable summands $($as $\mathbbm{k}Q$-module$)$
such that $f({\bf dim}M)=\beta$.
\end{proposition}

\begin{proof} Let $N$ be an indecomposable $\mathbbm{k}Q$-module and
$\alpha:={\bf dim}N$. Then $\sum(N)$ is an indecomposable
$G$-invariant $\mathbbm{k}Q$-module with dimension vector
$\sum_{g\in G_{N}}g(\alpha)$. We claim that
$$
\sum_{g\in G_{N}}g(\alpha)=m\Sigma(\alpha)
$$
for some positive integer $m$. Indeed, since $H_{N}\subseteq
H_{\alpha}$, we have $|H_{\alpha}|=m|H_{N}|$ and so that
$|G_{N}|=m|G_{\alpha}|$ for some positive integer $m$. Note that
$$
\sum_{g\in G_{\alpha}}g(\alpha)=\frac{\sum_{g\in
G}g(\alpha)}{|H_{\alpha}|} \quad\hbox{ and }\quad \sum_{g\in
G_{N}}g(\alpha)=\frac{\sum_{g\in G}g(\alpha)}{|H_{N}|},
$$
we obtain that
$$
{\bf dim}\sum(N)=\sum_{g\in
G_{N}}g(\alpha)=m\sum_{g\in G_{\alpha}}g(\alpha)=m\Sigma(\alpha).
$$
In particular, if $\alpha$ is a real root of $Q$, then
$H_{N}=H_{\alpha}$ and so that we  take $G_{N}=G_{\alpha}$ in this
case. Therefore, $f({\bf dim}\sum(N))\in \Delta_{\Gamma}$. Note that
for every indecomposable $G$-invariant $\mathbbm{k}Q$-module $M$,
there is an indecomposable $\mathbbm{k}Q$-module $N$ such that
$M\cong\sum(N)$, we get $f({\bf dim}M)\in\Delta_{\Gamma}$.

\medskip
If $\beta:=f({\bf dim}M)$ is a real root with $f({\bf
dim}M)=\omega'(\overline{\varepsilon}_{i})$ for some
$\omega'\in\mathcal{W}(\Gamma)$ and $i\in\mathcal{I}$, then ${\bf
dim}M=\omega(\Sigma(\varepsilon_{i'}))=\Sigma(\omega(\varepsilon_{i'}))$
for any $i'\in\mathcal{O}_{i}$, where $\omega\in
C_{G}(\mathcal{W}(Q))$ corresponding to $\omega'$, by the proof of
Lemma \ref{lem3-2}. Denote by $N$ the unique indecomposable
$\mathbbm{k}Q$-module with ${\bf dim}N=\omega(\varepsilon_{i'})$,
then $M=\sum(N)$ is the unique indecomposable $G$-invariant
$\mathbbm{k}Q$-module satisfying ${\bf
dim}M=\omega(\Sigma(\varepsilon_{i'}))$ and $M$ is independent on
the taking of $i'\in\mathcal{O}_{i}$. Finally, note that
$$
\frac{1}{2}({\bf dim}M, {\bf
dim}M)_{Q}=\frac{1}{2}(\Sigma(\varepsilon_{i'}),
\Sigma(\varepsilon_{i'}))_{Q}=d_{i}=|G_{\varepsilon_{i'}}|=|G_{N}|,
$$
we are done.
\end{proof}

We suppose now that $G$ is abelian and let
$$e:=\sum_{i\in\mathcal{I}}e_{i}
\in\mathbbm{k}Q\subseteq\mathbbm{k}Q\ast G,$$
where $e_{i}$ is the
idempotent element of $\mathbbm{k}Q$ corresponding to vertex $i\in
I$. By the proof of \cite[Theorem 1]{De}, we know that
$\mathbbm{k}Q\ast G$ is Morita equivalent to $e\mathbbm{k}Q\ast Ge$
and $e\mathbbm{k}Q\ast Ge\cong\mathbbm{k}\widehat{Q}$. Thus we view
the functor
$$
\begin{aligned}E:\quad\mbox{{\bf
mod}-}\mathbbm{k}Q\ast G\quad &\longrightarrow \quad\mbox{{\bf
mod}-}\mathbbm{k}\widehat{Q} \nonumber\\
M  \qquad&\mapsto  \qquad eM
\end{aligned}
$$
as the equivalence functor between ${\bf mod}$-$\mathbbm{k}Q\ast G$
and ${\bf mod}$-$\mathbbm{k}\widehat{Q}$. Denote by
$$
\begin{array}{ll} F:=\mathbbm{k}Q\ast G\otimes_{\mathbbm{k}Q}-:
\quad&\mbox{{\bf mod}-}\mathbbm{k}Q\quad \longrightarrow
\quad\mbox{{\bf mod}-}\mathbbm{k}Q\ast G\\
H:=\mbox{Res}|_{\mathbbm{k}Q}: & \mbox{{\bf mod}-}\mathbbm{k}Q\ast
G\quad \longrightarrow \quad\mbox{{\bf mod}-}\mathbbm{k}Q
\end{array}
$$
Following from \cite[Theorem 1.1]{De}, $(H, F)$ and $(F, H)$ are
adjoint pairs.

Moreover, for any $\mathbbm{k}\widehat{Q}$-module $X$, $HE^{-1}(X)$
is a $G$-invariant $\mathbbm{k}Q$-module and there is a
$\mathbbm{k}Q$-module $M$ such that $HE^{-1}(X)\cong\sum(M)$, where
$E^{-1}$ is the quasi-inverse of $E$. Identifying $X$ with a
$\widehat{Q}$-representation $(X_{i\rho}, X_{\alpha})$,  we have
$$
\sum_{\rho\in{\rm irr} G_{i}}X_{i\rho}\cong
e_{i}HE^{-1}(X)e_{i}\cong\bigoplus_{g\in G_{M}}(^{g}M)_{i}.
$$
Suppose ${\bf
dim}X:=\sum_{(i\rho)\in\widehat{I}}\alpha_{i\rho}\varepsilon_{i\rho}$,
then
$$
\sum_{\rho\in{\rm irr} G_{i}}\alpha_{i\rho}=\sum_{g\in
G_{M}}\dim(^{g}M)_{i}=f\Big({\bf dim}\sum(M)\Big)_{i}.
$$
Therefore, the Moriat equivalence and the restriction functor induce
a map
$$
h:\quad
\mathbb{Z}\widehat{I}\longrightarrow\mathbb{Z}\mathcal{I}
$$
given by $h(\alpha)_{i}=\sum_{\rho\in{\rm irr}
G_{i}}\alpha_{i\rho}\overline{\varepsilon}_{i}$ for any
$\alpha=\sum_{(i\rho)\in\widehat{I}}\alpha_{i\rho}\varepsilon_{i\rho}
\in\mathbb{Z}\widehat{I}$. The restriction of $h$ to the root system
$\Delta_{\widehat{Q}}$ is also denoted by $h$. Then $h:
\Delta_{\widehat{Q}}\rightarrow\Delta_{\Gamma}$ is well-defined
since $X$ is an indecomposable $\mathbbm{k}\widehat{Q}$-module if
and only if $M$ is an indecomposable $\mathbbm{k}Q$-module. By
Proposition \ref{prop3-3}, we have

\begin{corollary} \label{cor3-4}
For any indecomposable $\widehat{Q}$-representation $X$, $h({\bf
dim}X)$ is a positive root of $\Gamma$.
\end{corollary}

Up to now, we have obtained the map $h:
\mathbb{Z}\widehat{I}\rightarrow\mathbb{Z}\mathcal{I}$ and have
shown the half of Theorem \ref{thm1-1}(1). Before completing the proof
of Theorem \ref{thm1-1}, we should define an action of $G$ on
$\mathbbm{k}\widehat{Q}$ and give the dual between $(Q, G)$ and
$(\widehat{Q}, G)$. In the following subsection, we first describe
the duality of $(Q, G)$.

\medskip
\noindent {\bf 3.2.} We write the abelian group $G$ as the product
of some finite cyclic group, i.e.,
$$
G=\langle g_{1}\rangle\times\langle
g_{2}\rangle\times\cdots\times\langle g_{n}\rangle,
$$
where the
order of $g_{i}$ is $m_{i}$ for $1\leq i\leq n$. Then
$|G|=m_{1}m_{2}\cdots m_{n}$.

We now define an action of $G$ on $\widehat{Q}$. Since $G$ is
abelian, all the characters $\chi$ of $G$ are linear. The set of all
the characters of $G$ is an abelian group with the multiplication
$$
\chi\chi'(g)=\chi(g)\chi'(g),
$$
for all $g\in G$. We denote this group by $\widetilde{G}$. Setting
$\varphi: G\rightarrow\widetilde{G}$ by
$$
\varphi(g)=\chi_{g}, \quad
\chi_{g}(g')=\xi_{1}^{t_{1}s_{1}}\xi_{2}^{t_{2}s_{2}}
\cdots\xi_{n}^{t_{n}s_{n}}
$$
if $g=g_{1}^{t_{1}}g_{2}^{t_{2}}\cdots
g_{n}^{t_{n}}$ and $g'=g_{1}^{s_{1}}g_{2}^{s_{2}}\cdots
g_{n}^{s_{n}}$, where $\xi_{i}$ is a primitive $m_{i}$-th root of
unity. It is easy to see that $\varphi$ is a group isomorphism. By
\cite{RR}, we define a linear action of $G$ on $\mathbbm{k}Q\ast G$
by setting
$$
g(\lambda h)=\chi_{g}(h)\lambda h,
$$
for all $g\in G, \lambda h\in\mathbbm{k}Q\ast G$. Then
$G\subseteq\Aut(\mathbbm{k}Q\ast G)$. By \cite[Proposition 5.1]{RR},
we have

\begin{proposition} \label{prop3-5}
The map $\psi : (\mathbbm{k}Q\ast G)\ast G\rightarrow{\rm
End}_{\mathbbm{k}Q}(\mathbbm{k}Q\ast G)$ defined by
$$
\psi(\lambda gh)(\mu h')=\chi_{h}(h')\lambda g\mu h'
$$
is an algebra isomorphism. It follows that $(\mathbbm{k}Q\ast G)\ast
G$ is Morita equivalent to $\mathbbm{k}Q$.
\end{proposition}

Since $e\mathbbm{k}Q\ast Ge\cong\mathbbm{k}\widehat{Q}$ and the
action of $G$ on $\mathbbm{k}Q\ast G$ stabilizes $e$, the action of
$G$ on $\mathbbm{k}Q\ast G$ naturally induces an action of $G$ on
$\mathbbm{k}\widehat{Q}$ such that
$G\subseteq\Aut(\mathbbm{k}\widehat{Q})$. Therefore, we get a skew
group algebra $\mathbbm{k}\widehat{Q}\ast G$ under this action. Let
$\widehat{\widehat{Q}}$ be the generalized Mckay quiver of
$(\widehat{Q}, G)$. Then, there is a Morita equivalence between
$\mathbbm{k}\widehat{Q}\ast G$ and
$\mathbbm{k}\widehat{\widehat{Q}}$ by Theorem \ref{thm2-1}.

\begin{proposition} \label{prop3-6}
Let $\widehat{Q}$ be the generalized McKay quiver of $(Q, G)$ under
the action of $G$ defined as above. Then the generalized McKay
quiver $\widehat{\widehat{Q}}$ of $(\widehat{Q}, G)$ coincides with
$Q$.
\end{proposition}

Thus we get the dual between $(Q, G)$  and $(\widehat{Q}, G)$. Now,
for the relationship between quivers $Q$ and $\widehat{Q}$, and the
action of $G$ on $\widehat{Q}$, we give some more description. Note
that the stabilizer $G_{i}$ of $i\in I$ has the form
$$
G_{i}=\langle
g_{1}^{d_{i_{1}}}\rangle\times\langle
g_{2}^{d_{i_{2}}}\rangle\times\cdots\times\langle
g_{n}^{d_{i_{n}}}\rangle,$$ where
$$\nu_{i_{j}}:=|\langle
g_{j}^{d_{i_{j}}}\rangle|=\frac{m_{i}}{d_{i_{j}}}, \qquad 1\leq
j\leq n,
$$
and so that
$$
d_{i}=|\mathcal{O}_{i}|=\frac{|G|}{|G_{i}|}=d_{i_{1}}\times
d_{i_{2}}\times\cdots\times d_{i_{n}}.
$$
We set
$$
\begin{aligned}&e_{(i,s_{i_{1}},s_{i_{2}},\cdots,s_{i_{n}})}\nonumber\\
&\qquad=\frac{1}{|G_{i}|}
\sum_{j_{1}=0}^{\nu_{i_{1}}-1}\sum_{j_{2}=0}^{\nu_{i_{2}}-1}\cdots
\sum_{j_{n}=0}^{\nu_{i_{n}}-1}\xi_{1}^{d_{i_{1}}j_{1}s_{i_{1}}}
\xi_{2}^{d_{i_{2}}j_{2}s_{i_{2}}}\cdots\xi_{n}^{d_{i_{n}}j_{n}s_{i_{n}}}
g_{1}^{d_{i_{1}}j_{1}}g_{2}^{d_{i_{2}}j_{2}}\cdots
g_{n}^{d_{i_{n}}j_{n}}.
\end{aligned}
$$
Then one can check that
$\big\{e_{(i,s_{i_{1}},s_{i_{2}},\cdots,s_{i_{n}})}\mid
s_{i_{j}}\in\mathbb{Z}/\nu_{i_{j}}\mathbb{Z}\mbox{ for all } 1\leq
j\leq n\big\}$ is a complete set of primitive orthogonal idempotents
of $\mathbbm{k}[G_{i}]$.

It is obvious that
$$
g_{j}(e_{(i,s_{i_{1}},s_{i_{2}},\cdots,s_{i_{n}})})
=e_{(i,s_{i_{1}},\cdots,s_{i_{j-1}},s'_{i_{j}},s_{i_{j+1}},
\cdots,s_{i_{n}})}
$$
for any $1\leq j\leq n,$ where
$s'_{i_{j}}\in\mathbb{Z}/\nu_{i_{j}}\mathbb{Z}$ and
$s'_{i_{j}}\equiv s_{i_{j}}+1 \mod{\nu_{i_{j}}}$. Since for each
idempotent $e_{(i,s_{i_{1}},s_{i_{2}},\cdots,s_{i_{n}})}$, there is
a unique corresponding one dimensional irreducible representation
$\rho$ of $G_{i}$ defined by the group homomorphism $\phi_{\rho} :
G_{i}\rightarrow\mathbbm{k}$,
$g_{j}^{d_{i_{j}}}\mapsto\xi^{d_{i_{j}}s_{i_{j}}}$, for $1\leq j\leq
n$. Thus we can index the vertices set $\widehat{I}$ by some
sequences $(i,s_{i_{1}},s_{i_{2}},\cdots,s_{i_{n}})$, i.e.,
$$
\widehat{I}=\left\{(i,s_{i_{1}},s_{i_{2}},\cdots,s_{i_{n}})\mid
i\in\mathcal{I}, s_{i_{j}}\in\mathbb{Z}/\nu_{i_{j}}\mathbb{Z} \mbox{
for all } 1\leq j\leq n\right\}.
$$
Then the action of $G$ on $\widehat{I}$ is clearly and so that the
orbit of $(i, \rho)\in\widehat{I}$ has the form
$$
\left\{(i,s_{i_{1}},s_{i_{2}},\cdots,s_{i_{n}})\mid
s_{i_{j}}\in\mathbb{Z}/\nu_{i_{j}}\mathbb{Z} \mbox{ for } 1\leq
j\leq n\}=\{(i, \rho)\mid \rho\in\irr G_{i}\right\}
$$
for some
$i\in\mathcal{I}$. Furthermore, it is easy to see that if the action
of $G$ on $\mathbbm{k}Q$ is admissible, then so is on
$\mathbbm{k}\widehat{Q}$.

For any $i, j\in I$, we consider the group $G_{ij}:=G_{i}\cap
G_{j}=\langle g_{1}^{t_{1}}\rangle\times\langle
g_{2}^{t_{2}}\rangle\times\cdots\times\langle g_{n}^{t_{n}}\rangle$,
where $t_{l}$ is the least common multiple of $d_{i_{l}}$ and
$d_{j_{l}}$ for $1\leq l\leq n$. Note that the vector space $E_{ij}$
spanned by arrows $\alpha: i\rightarrow j$ in $Q$ is a
$\mathbbm{k}[G_{ij}]$-bimodule, we can find a basis of $E_{ij}$ such
that the action of $G_{ij}$ is diagonal. That is, if
$g=g_{1}^{t_{1}}g_{2}^{t_{2}}\cdots
g_{n}^{t_{n}}\in\mathbbm{k}[G_{ij}]$, then for any basis element
$\alpha'\in E_{ij}$,
$$g(\alpha')=\xi_{1}^{t_{1}r_{1}}
\xi_{2}^{t_{2}r_{2}}\cdots\xi_{n}^{t_{n}r_{n}}\alpha'$$ for some
$r_{1},r_{2},\cdots, r_{n}\in\mathbb{Z}$. Since $G_{ij}$ is abelian,
the number of the basis elements of $E_{ij}$ is just the number of
arrows from $i$ to $j$ in $Q$. Moreover, it is easy to see that the
$t_{1}t_{2}\cdots t_{n}$ elements
$$\alpha', ~g_{1}(\alpha'),\cdots, ~g_{n}(\alpha'),
~g^{2}_{1}(\alpha'), ~g_{1}g_{2}(\alpha'),\cdots,
~g^{2}_{n}(\alpha'),\cdots\cdots,
~g^{t_{1}-1}_{1}g^{t_{2}-1}_{2}\cdots g^{t_{n}-1}_{n}(\alpha')$$ are
linearly independent. That is,  for any arrow $\alpha:i\rightarrow
j$ in $Q$, there are $t_{1} t_{2}\cdots t_{n}$ arrows in its orbit.

On the other hand, we can calculate that
$$
\begin{aligned}&e_{(j,s_{j_{1}},s_{j_{2}},\cdots,s_{j_{n}})}\alpha'
e_{(i,s_{i_{1}},s_{i_{2}},\cdots,s_{i_{n}})}\nonumber\\
&\qquad=\frac{d_{i}d_{j}}{|G|^{2}}\sum_{p_{1}=0}^{\nu_{i_{1}}-1}\cdots
\sum_{p_{n}=0}^{\nu_{i_{n}}-1}
\sum_{q_{1}=0}^{\nu_{j_{1}}-1}\cdots\sum_{q_{n}=0}^{\nu_{j_{n}}-1}
\xi_{1}^{d_{i_{1}}p_{1}s_{i_{1}}+d_{j_{1}}q_{1}s_{j_{1}}}\cdots
\xi_{n}^{d_{i_{n}}p_{n}s_{i_{n}}+d_{j_{n}}q_{n}s_{j_{n}}}\nonumber\\
&\qquad\qquad\qquad\qquad\qquad\qquad\qquad\quad
g_{1}^{d_{j_{1}}q_{1}}\cdots g_{n}^{d_{j_{n}}q_{n}}(\alpha')
g_{1}^{d_{i_{1}}p_{1}+d_{j_{1}}q_{1}}\cdots
g_{n}^{d_{i_{n}}p_{n}+d_{j_{n}}q_{n}}.
\end{aligned}
$$
We write
\begin{eqnarray*}
d_{i_{l}}p_{l}&=&P_{l}t_{l}+d_{i_{l}}p'_{l}, \qquad\hbox{ where }
0\leq
P_{l}<\frac{m_{l}}{t_{l}}, \quad 0\leq p'_{l}<\frac{t_{l}}{d_{i_{l}}},\\
d_{j_{l}}q_{l}&=&P'_{l}t_{l}+d_{j_{l}}q'_{l}, \qquad\hbox{ where }
0\leq
P'_{l}<\frac{m_{l}}{t_{l}}, \quad 0\leq q'_{l}<\frac{t_{l}}{d_{j_{l}}},\\
d_{i_{l}}k_{l}&\equiv& (P_{l}+P'_{l})t_{l}+d_{i_{l}}p'_{l}
\mod{m_{i}}, \qquad\hbox{ where } 0\leq k_{l}<\nu_{i_{l}},
\end{eqnarray*}
for all $0\leq l\leq n$. Then the right side of the equation becomes
$$
\begin{aligned}\frac{d_{i}d_{j}}{|G|^{2}}
&\Bigg(\sum_{P'_{1}=0}^{\frac{m_{1}}{t_{1}}-1}
\xi_{1}^{P'_{1}t_{1}(r_{1}+s_{j_{1}}-s_{i_{1}})}\Bigg)\cdots
\Bigg(\sum_{P'_{n}=0}^{\frac{m_{n}}{t_{n}}-1}
\xi_{n}^{P'_{n}t_{n}(r_{n}+s_{j_{n}}-s_{i_{n}})}\Bigg)\nonumber\\
&\quad
\Bigg(\sum_{k_{1}=0}^{\nu_{i_{1}}-1}\cdots\sum_{k_{n}=0}^{\nu_{i_{n}}-1}
\sum_{q'_{1}=0}^{\frac{t_{1}}{d_{j_{1}}}-1}\cdots
\sum_{q'_{n}=0}^{\frac{t_{n}}{d_{j_{n}}}-1}
\xi_{1}^{d_{i_{1}}k_{1}s_{i_{1}}+d_{j_{1}}q'_{1}s_{j_{1}}}\cdots
\xi_{n}^{d_{i_{n}}k_{n}s_{i_{n}}+d_{j_{n}}q'_{n}s_{j_{n}}}\nonumber\\
&\qquad\qquad\qquad\qquad\qquad\qquad g_{1}^{d_{j_{1}}q'_{1}}\cdots
g_{n}^{d_{j_{n}}q'_{n}}(\alpha')
g_{1}^{d_{i_{1}}k_{1}+d_{j_{1}}q'_{1}}\cdots
g_{n}^{d_{i_{n}}k_{n}+d_{j_{n}}q'_{n}}\Bigg).
\end{aligned}
$$
Note that
$$
\begin{aligned}\Big\{g_{1}^{d_{j_{1}}q'_{1}}\cdots
g_{n}^{d_{j_{n}}q'_{n}}(\alpha')
g_{1}^{d_{i_{1}}k_{1}+d_{j_{1}}q'_{1}}&\cdots
g_{n}^{d_{i_{n}}k_{n}+d_{j_{n}}q'_{n}}\nonumber\\
& \mid 0\leq k_{l}<\nu_{i_{l}}, ~0\leq
q'_{l}<\frac{t_{l}}{d_{j_{l}}}\mbox{ for } 1\leq l\leq n\Big\}
\end{aligned}
$$
is a linearly independent set. We obtain that
$e_{(j,s_{j_{1}},s_{j_{2}},\cdots,s_{j_{n}})}\alpha'
e_{(i,s_{i_{1}},s_{i_{2}},\cdots,s_{i_{n}})}\neq0$ if and only if
$s_{i_{l}}\equiv s_{j_{l}}+r_{l} \mod{\frac{m_{l}}{t_{l}}}$ for all
$0\leq l\leq n$. It follows that there are $\frac{t_{1}\cdots
t_{n}|G|}{d_{i}d_{j}}$ arrows in $\widehat{Q}$ for each arrow
$\alpha:i\rightarrow j$ in $Q$.

\medskip
Denote by
$\widehat{A}=(a_{(i\rho)(j\sigma)})_{\widehat{I}\times\widehat{I}}$
the Cartan matrix of $\widehat{Q}$, by $\widehat{\Gamma}$ the valued
quiver corresponding to $(\widehat{Q}, G)$ and by
$\widehat{C}=(\widehat{c}_{ij})_{\mathcal{I}\times\mathcal{I}}
=\widehat{D}^{-1}\widehat{B}$ the generalized Cartan matrix of
$\widehat{\Gamma}$, where
$\widehat{B}=(\widehat{b}_{ij})_{\mathcal{I}\times\mathcal{I}}$ is
symmetric, $\widehat{D}=\mbox{diag}(\widehat{d}_{i})$ is diagonal.
Then
$$
\frac{1}{t_{1}\cdots t_{n}}\sum_{i'\in\mathcal{O}_{i}\atop
j'\in\mathcal{O}_{j}}a_{i'j'}=\frac{d_{i}d_{j}}{t_{1}\cdots
t_{n}|G|}\sum_{\rho\in{\rm irr} G_{i}\atop \sigma\in{\rm irr}
G_{j}}a_{(i\rho)(j\sigma)}.
$$
It follows that $\widehat{b}_{ij}=\frac{|G|}{d_{i}d_{j}}b_{ij}$,
$\widehat{D}=|G|D^{-1}$, $\widehat{B}=|G|D^{-1}BD^{-1}$ and
$\widehat{C}=(\widehat{D})^{-1}\widehat{B}=BD^{-1}=C^{T}$, the
transpose of $C$. Therefore $\Gamma$ and $\widehat{\Gamma}$ are dual
valued graph in the sense of \cite{Ka}.

\begin{remark} \label{rem3-7}
If $G\subseteq\Aut(\mathbbm{k}Q)$ is a finite abelian group, we have
given the dual of $(Q, G)$ and $(\widehat{Q}, G)$ (see Proposition
\ref{prop3-6}). However, for a non-abelian group
$G\subseteq\Aut(\mathbbm{k}Q)$, the conclusion does not hold in
general. For example, Let $Q$ be the quiver
$$~~~~~~~~~~~~~~~~~~~~~~~\setlength{\unitlength}{1mm}
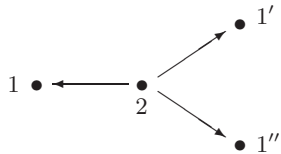
\begin {picture}(45,20)
\put(22,8){$\bullet$}          \put(25,10){\vector (3,2){9}}
\put(35,16){$\bullet$}         \put(25,8){\vector (3,-2){9}}
\put(35,0){$\bullet$}           \put(21,9){\vector (-1,0){10}}
\put(8,8){$\bullet$}            \put(5,8){\small$1$}
\put(38,17){\small$1'$}          \put(38,0){\small$1''$}
\put(22,5){\small$2$}
\end {picture}
$$
It is well-known that the quiver automorphism group of $Q$ is
the group $S_{3}$. Accordingly, we obtain
the generalized McKay quiver $\widehat{Q}$ of $(Q, S_{3})$ as
follows
$$~~~~~~~~~~~~~~~~~~~~~~~\setlength{\unitlength}{1mm}
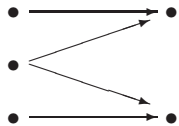
\begin {picture}(45,17)
\put(13,14){$\bullet$}              \put(13,7){$\bullet$}
\put(13,0){$\bullet$}               \put(16,15){\vector (1,0){17}}
\put(16,1){\vector (1,0){17}}        \put(34,14){$\bullet$}
\put(34,0){$\bullet$}              \put(16,8){\vector (3,-1){16}}
\put(16,8.5){\vector (3,1){16}}
\end{picture}
$$
One can check that there does not exist
a subgroup $G'$ of $\Aut(\mathbbm{k}\widehat{Q})$ such that the
generalized McKay quiver of $(\widehat{Q}, G')$ is $Q$.

\medskip
But if the action of $G$ is ``good", there exists the duality still.
For example, we consider the finite non-abelian group
$$
G=\left\langle a, b\mid a^{3}=b^{2}, b^{4}=1, aba=b\right\rangle
$$
and the quiver $Q$:
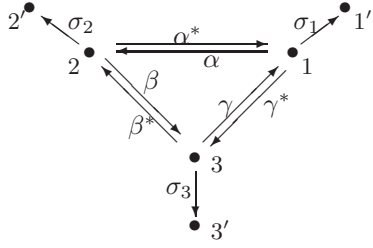
$$~~~~~~~~~~~~~~~~~~~~~~~\setlength{\unitlength}{1mm}
\begin{picture}(50,32)
\put(24,1){$\bullet$}                \put(24,10){$\bullet$}
\put(37,24){$\bullet$}               \put(10, 24){$\bullet$}
\put(44,30){$\bullet$}               \put(2, 30){$\bullet$}
\put(25,9){\vector (0,-1){6}}        \put(26,13){\vector (1,1){10}}
\put(22.5,13){\vector (-1,1){10}}    \put(13,24){\vector (1,-1){10}}
\put(37,22.5){\vector (-1,-1){10}}  \put(14.5,26){\vector (1,0){20}}
\put(34.5,25){\vector (-1,0){20}}   \put(9.5,26){\vector (-4,3){5}}
\put(39,26){\vector (4,3){5}}       \put(22,26){$\alpha^{\ast}$}
\put(26,23){$\alpha$}                   \put(18,20){$\beta$}
\put(16,14){$\beta^{\ast}$}              \put(28,17){$\gamma$}
\put(34,17){$\gamma^{\ast}$}             \put(39,22){\small $1$}
\put(46,28){\small $1'$}                 \put(8,22){\small $2$}
\put(0,28){\small $2'$}                 \put(27,9){\small $3$}
\put(27,0){\small $3'$}                  \put(21,6){$\sigma_{3}$}
\put(8,28){$\sigma_{2}$}                \put(38,28){$\sigma_{1}$}
\end{picture}.$$
The action of $G$ is given by
\begin{center}
\begin{tabular*}{8cm}{@{\extracolsep{\fill}}l|llllllllllllllr}
& $e_{1}$ & $e_{2}$ & $e_{3}$ & $e_{1'}$ & $e_{2'}$ & $e_{3'}$ &
$\alpha$ & $\alpha^{\ast}$  \\
\hline $a$ & $e_{2}$ & $e_{3}$ & $e_{1}$ & $e_{2'}$ & $e_{3'}$ &
$e_{1'}$ & $\beta$ & $\beta^{\ast}$ \\
$b$ & $e_{1}$ & $e_{3}$ & $e_{2}$ & $e_{1'}$ & $e_{2'}$ & $e_{2'}$ &
$-\gamma^{\ast}$ & $\gamma$ \\
\end{tabular*}

\medskip
\begin{tabular*}{6cm}{@{\extracolsep{\fill}}l|lllllllllllr}
& $\beta$ & $\beta^{\ast}$ & $\gamma$ & $\gamma^{\ast}$ &
$\sigma_{1}$ & $\sigma_{2}$ & $\sigma_{3}$ \\
\hline $a$ & $\gamma$ & $\gamma^{\ast}$ & $\alpha$ & $\alpha^{\ast}$
& $\sigma_{2}$ & $\sigma_{3}$
& $\sigma_{1}$ \\
$b$ & $-\beta^{\ast}$ & $\beta$ & $-\alpha^{\ast}$ & $\alpha$ &
$\sigma_{1}$ & $\sigma_{3}$
& $\sigma_{2}$ \\
\end{tabular*}
\end{center}
where $e_{i}$ is the idempotent element of $\mathbbm{k}Q$
corresponding to vertex $i$, $i\in\{1, 2, 3, 1', 2',$ $3'\}$. By
direct calculation, one see that the generalized McKay quiver of
$(Q, G)$ is as follows.
$$~~~~~~~~~~~~~~~~~~~~~~~\setlength{\unitlength}{1mm}
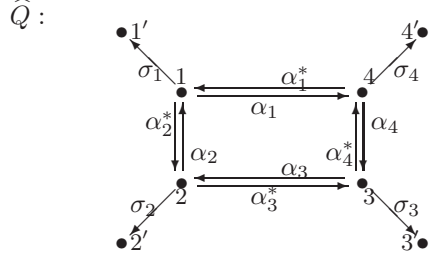
\begin{picture}(50,34)
\put(17,8){$\bullet$}                     \put(41,8){$\bullet$}
\put(17,20){$\bullet$}                    \put(41,20){$\bullet$}
\put(49,28){$\bullet$}                   \put(9,28){$\bullet$}
\put(49,0){$\bullet$}                    \put(9,0){$\bullet$}
\put(20,20.5){\vector (1,0){20}} \put(39.5,21.5){\vector (-1,0){20}}
\put(20,8.5){\vector (1,0){20}}  \put(39.5,9.5){\vector (-1,0){20}}
\put(18,10.5){\vector (0,1){9}}    \put(17,19.5){\vector (0,-1){9}}
\put(41,10.5){\vector (0,1){9}}   \put(42,19.5){\vector (0,-1){9}}
\put(17,22){\vector (-1,1){6}}     \put(17,8){\vector (-1,-1){6}}
\put(43,22){\vector (1,1){6}}     \put(43,8){\vector (1,-1){6}}
\put(-5,30){$\widehat{Q} :$}     \put(31,22){$\alpha_{1}^{\ast}$}
\put(27,18){$\alpha_{1}$}         \put(31,10){$\alpha_{3}$}
\put(27,6){$\alpha_{3}^{\ast}$}   \put(13,16){$\alpha_{2}^{\ast}$}
\put(19,12){$\alpha_{2}$}             \put(43,16){$\alpha_{4}$}
\put(37,12){$\alpha_{4}^{\ast}$}      \put(17,22){\small $1$}
\put(42,22){\small $4$}               \put(17,6){\small $2$}
\put(42,6){\small $3$}                \put(11,28){\small $1'$}
\put(47,28){\small $4'$}              \put(11,0){\small $2'$}
\put(47,0){\small $3'$}               \put(12,23){$\sigma_{1}$}
\put(11,5){$\sigma_{2}$}              \put(46,5){$\sigma_{3}$}
\put(46,23){$\sigma_{4}$}
\end{picture}$$
Now, we define an action of $G$ on $\mathbbm{k}\widehat{Q}$ by
setting
\begin{center}
\begin{tabular*}{11cm}{@{\extracolsep{\fill}}l|llllllllllllllr}
& $e_{1}$ & $e_{2}$ & $e_{3}$ & $e_{4}$ & $e_{1'}$ & $e_{2'}$ &
$e_{3'}$ & $e_{4'}$ & $\alpha_{1}$ & $\alpha_{2}$ & $\alpha_{3}$ &
$\alpha_{4}$ \\
\hline $a$ & $e_{3}$ & $e_{4}$ & $e_{1}$ & $e_{2}$ & $e_{3'}$ &
$e_{4'}$ & $e_{1'}$ & $e_{4'}$ & $\xi^{2}\alpha_{3}$ &
$\xi^{4}\alpha_{4}$ &
$\xi^{2}\alpha_{1}$ & $\xi^{4}\alpha_{2}$  \\
$b$ & $e_{2}$ & $e_{3}$ & $e_{4}$ & $e_{1}$ & $e_{2'}$ & $e_{3'}$ &
$e_{4'}$ & $e_{1'}$ & $\alpha_{2}$ & $\alpha_{3}$ & $\alpha_{4}$ &
$\alpha_{1}$ \\
\end{tabular*}

\medskip
\begin{tabular*}{9cm}{@{\extracolsep{\fill}}l|llllllllllllllr}
& $\alpha^{\ast}_{1}$ & $\alpha^{\ast}_{2}$ & $\alpha^{\ast}_{3}$ &
$\alpha^{\ast}_{4}$ & $\sigma_{1}$ & $\sigma_{2}$
& $\sigma_{3}$ & $\sigma_{4}$ \\
\hline $a$ & $\xi^{2}\alpha^{\ast}_{3}$ & $\xi^{4}\alpha^{\ast}_{4}$
& $\xi^{2}\alpha^{\ast}_{1}$ & $\xi^{4}\alpha^{\ast}_{2}$ &
$\sigma_{3}$ & $\sigma_{4}$ & $\sigma_{1}$ & $\sigma_{2}$ \\
$b$ & $\alpha^{\ast}_{2}$ & $\alpha^{\ast}_{3}$ &
$\alpha^{\ast}_{4}$ & $\alpha^{\ast}_{1}$ & $\sigma_{2}$ &
$\sigma_{3}$ & $\sigma_{4}$ & $\sigma_{1}$  \\
\end{tabular*}
\end{center}
where $\xi$ is a primitive $6$-th root of unity. Then, one can check
that $\widehat{\widehat{Q}}=Q$.
\end{remark}

\noindent {\bf 3.3.} Consider the admissible action of finite
abelian group $G$ on $\mathbbm{k}\widehat{Q}$ induced from the
action of $G$ on $\mathbbm{k}Q$ as the discussion above, we set
$$
\begin{array}{ll} F':=(\mathbbm{k}Q\ast G)\ast G
\otimes_{\mathbbm{k}Q\ast G}-: \quad&\mbox{{\bf
mod}-}\mathbbm{k}Q\ast G \longrightarrow
\mbox{{\bf mod}-}(\mathbbm{k}Q\ast G)\ast G\\
H':=\mbox{Res}|_{\mathbbm{k}Q\ast G}:& \mbox{{\bf
mod}-}(\mathbbm{k}Q\ast G)\ast G\longrightarrow \mbox{{\bf
mod}-}\mathbbm{k}Q\ast G
\end{array}
$$
Similar to the functors $F$ and $H$, one can check that $(H', F')$ and
$(F', H')$ are adjoint pairs. Note that the Morita equivalence
$\mbox{{\bf mod}-}\mathbbm{k}Q\rightarrow\mbox{{\bf
mod}-}(\mathbbm{k}Q\ast G)\ast G$ is given by
$\mathcal{M}:={_{(\mathbbm{k}Q\ast G)\ast G}\mathbbm{k}Q\ast
G}\otimes_{\mathbbm{k}Q}-$, we have

\begin{lemma} \label{lem3-8}
There are natural isomorphisms
$$
F\cong H'\mathcal{M}\qquad\mbox{
and }\qquad F'\cong\mathcal{M}H.
$$
\end{lemma}

\begin{proof}
First, $H'\mathcal{M}={_{\mathbbm{k}Q\ast G}\mathbbm{k}Q\ast
G}\otimes_{\mathbbm{k}Q}-=F$ is clear. Next, since $(H', F')$ is
an adjoint pair, for any $\mathbbm{k}Q$-module $X$ and
$\mathbbm{k}Q\ast G$-module $Y$, we have
$$
\begin{aligned}\Hom_{\mathbbm{k}Q}(X,
\mathcal{M}^{-1}F'(Y))&\cong\Hom_{(\mathbbm{k}Q\ast G)\ast
G}(\mathcal{M}(X), F'(Y))\nonumber\\
&\cong\Hom_{\mathbbm{k}Q\ast G}(H' \mathcal{M}(X),
Y)\cong\Hom_{\mathbbm{k}Q\ast G}(F(X), Y).
\end{aligned}
$$
This implies that $(F, \mathcal{M}^{-1}F')$ is an adjoint pair and
so that $H\cong\mathcal{M}^{-1}F'$, $F'\cong\mathcal{M}H$.
\end{proof}

By Lemma \ref{lem3-8} and \cite[Proposition 1.8]{RR}, we have the following
proposition immediately.

\begin{proposition} \label{prop3-9}
Let $X$ and $Y$ be indecomposable $\mathbbm{k}Q\ast G$-modules. Then

(1) $FH(X)\cong H'F'(X)\cong\bigoplus_{g\in G}{^{g}X}$;

(2) $H(X)\cong H(Y)$ if and only if $F'(X)\cong F'(Y)$, if and only
if $Y\cong{^{g}X}$ for some $g\in G$;

(3) $H(X)$ (or $F'(X)$) has exactly $|H_{X}|$ indecomposable
summands.
\end{proposition}

\begin{remark} \label{rem3-10}
Consider the action of $G$ on $\mathbbm{k}Q\ast G$, we denote by
$$H_{X}:=\{g\in G\mid F_{g}(X)\cong X\}$$
 and by $G_{X}$ a complete
set of left coset representatives of $H_{X}$ in $G$, for any
$X\in\mbox{{\bf mod}-}\mathbbm{k}Q\ast G$. In \cite{HY}, we have
shown that the number of indecomposable summands of $F'(X)$ is just
$|H_{X}|$ whenever $G$ is abelian (see \cite[Theorem 1.2]{HY}). This
means that $H(X)$ has $|H_{X}|$ indecomposable summands. Note that
$H(X)$ is an indecomposable $G$-invariant $\mathbbm{k}Q$-module,
there exists a unique indecomposable $\mathbbm{k}Q$-module $M$ such
that $H(X)\cong\sum(M)$. Therefore, we have $|H_{X}|=|G_{M}|$ and
$|G_{X}|=|H_{M}|$. Following from Proposition \ref{prop3-9}(2), for
an indecomposable $\mathbbm{k}Q$-module $M$, there are
$|G_{X}|=|H_{M}|$ non-isomorphic indecomposable $\mathbbm{k}Q\ast
G$-module structures on $\sum(M)$. This coincides with the result in
\cite{HY}. \end{remark}

For the generalized McKay quiver $\widehat{Q}$, we denote by $(-,
-)_{\widehat{Q}}$ the bilinear form on $\mathbb{Z}\widehat{I}$
determined by $\widehat{A}$, by $\Delta_{\widehat{Q}}$ the root
system of $\widehat{Q}$ with simple roots $\varepsilon_{i\rho}$,
$(i, \rho)\in\widehat{I}$, and by $\mathcal{W}(\widehat{Q})$ the
Weyl group of $\widehat{Q}$ with simple reflections $r_{i\rho}$,
$(i, \rho)\in\widehat{I}$. Consider the map $h :
\mathbb{Z}\widehat{I}\rightarrow\mathbb{Z}\mathcal{I}$ defined
above, we have

\begin{lemma} \label{lem3-11}
Let $\widehat{S}_{i}:=\prod_{\rho\in {\rm irr} G_{i}}r_{i\rho}$ for
$i\in\mathcal{I}$. Then, for each $i\in\mathcal{I}$ and
$\beta=\sum_{(i,
\rho)\in\widehat{I}}\beta_{i\rho}\varepsilon_{i\rho}
\in\mathbb{Z}\widehat{I}$, we have

(1) $(h(\beta),
\overline{\varepsilon}_{i})_{\Gamma}=d_{i}\sum_{\rho\in {\rm irr}
G_{i}}(\beta, \varepsilon_{i\rho})_{\widehat{Q}}$;

(2) $h(\widehat{S}_{i}(\beta))=\gamma_{i}(h(\beta))$;

(3) the map $\gamma_{i}\mapsto\widehat{S}_{i}$ induces an
isomorphism
$\mathcal{W}(\Gamma)\stackrel{\simeq}{\longrightarrow}C_{G}(\mathcal
{W}(\widehat{Q}))$, the set of elements in
$\mathcal{W}(\widehat{Q})$ commuting with the action of $G$.
\end{lemma}

\begin{proof} (1) By the dual between $(Q, G)$ and $(\widehat{Q}, G)$,
we obtain that
$$
b_{ij}=\sum_{i'\in\mathcal{O}_{i}\atop
j'\in\mathcal{O}_{j}}a_{i'j'}=\frac{d_{i}d_{j}}{|G|}\sum_{\rho\in{\rm
irr} G_{i}\atop\sigma\in{\rm irr} G_{j}}a_{(i\rho)(j\sigma)},
$$
and so that
$$
b_{ij}=d_{i}\sum_{\rho\in{\rm irr} G_{i}}a_{(i\rho)(j\sigma)}
$$
for any $\sigma\in{\rm irr} G_{j}$. Therefore, we get
$$
(h(\beta),
\overline{\epsilon}_{i})_{\Gamma}=\sum_{i,j\in\mathcal{I}}
b_{ij}h(\beta)_{j}=d_{i}\sum_{\rho\in{\rm irr} G_{i}\atop
\sigma\in{\rm irr}
G_{j}}a_{(i\rho)(j\sigma)}\beta_{j\sigma}=d_{i}\sum_{\rho\in{\rm
irr} G_{i}}(\beta, \varepsilon_{i\rho})_{\widehat{Q}}.
$$

(2) Firstly, $\widehat{S}_{i}$ is well-defined since the action of
$G$ on $\widehat{Q}$ is admissible. Secondly, it is easy to check
that the bilinear form $(-,-)_{\widehat{Q}}$ is $G$-invariant and
$\widehat{S}_{i}$ commutes with the action of $G$. Thus, we have
$$
h(\widehat{S}_{i}(\beta))=h(\beta)-\sum_{\rho\in{\rm irr}
G_{i}}(\beta,
\varepsilon_{i\rho})_{\widehat{Q}}\overline{\varepsilon}_{i}
=h(\beta)-\frac{1}{d_{i}}(h(\beta),
\overline{\varepsilon}_{i})_{\Gamma}\overline{\varepsilon}_{i}
=\gamma_{i}(h(\beta)).
$$

(3) By induction on the length, one can check that $C_{G}(\mathcal
{W}(\widehat{Q}))$ is generated by $\widehat{S}_{i}$, $i\in\mathcal
{I}$. Following from (2), we get $\gamma_{i}\mapsto S_{i}$ induces
an isomorphism.
\end{proof}

We are in a position to complete the proof of Theorem \ref{thm1-1}.
We have shown that for any positive root $\alpha\in\Delta_{\Gamma}$,
there exists an indecomposable $\widehat{Q}$-representation X such
that $h({\bf dim}X)=\alpha$. Moreover, if $\alpha$ is real, the
number of $X$ (up to isomorphism) can be determined. Applying the
technique in \cite[Proposition 15]{Hu1}, we have

\begin{proposition} \label{prop3-12}
The map $h : \Delta_{\widehat{Q}}\rightarrow\Delta_{\Gamma}$ is a
surjection. If $\alpha\in\Delta_{\Gamma}$ is a positive real root,
then there is a unique $G$-orbit of roots mapping to $\alpha$, and
all of which are real.
\end{proposition}

\begin{proof} Firstly, by Corollary \ref{cor3-4}, the map $h :
\Delta_{\widehat{Q}}\rightarrow\Delta_{\Gamma}$ is well-defined. To
show the surjectivity, we need to fine the preimages of all the
fundamental roots in $\Delta_{\widehat{Q}}$.

We suppose that $\Gamma$ is connected. Then, for any $\alpha\in
F_{\Gamma}$, we consider the set
$$
{\mathcal R}:=\{\beta\in\Delta_{\widehat{Q}}\mid\beta \mbox{ is
positive and } h(\beta)\leq\alpha\}.
$$
Since ${\mathcal R}$ is finite and non-empty, we take an element
$\beta$ with maximal height. Suppose that $h(\beta)_{i}<\alpha_{i}$
for all $i\in\mathcal{I}$, then for any $\rho\in\irr G_{i}$,
$h(\beta+\varepsilon_{i\rho})=h(\beta)+\overline{\varepsilon}_{i}\leq\alpha$.
By the maximality of $\beta$, $\beta+\varepsilon_{i\rho}$ is not a
root and so that $(\beta, \varepsilon_{i\rho})_{\widehat{Q}}\geq0$.
Thus $(h(\beta), \overline{\varepsilon}_{i})_{\Gamma}\geq0$ for all
$i\in\mathcal{I}$. We conclude that $h(\beta)$ and $\alpha$ have the
same support, for otherwise, we can find such a vertex $(i, \rho)$
adjacent to the support of $\beta$ such that $(\beta,
\varepsilon_{i\rho})_{\widehat{Q}}<0$.

We take $\alpha\in F_{\Gamma}$ such that the support of $\alpha$ is
${\mathcal I}$, and set
$$
\Phi:=\{i\in\mathcal{I}\mid
h(\beta)_{i}=\alpha_{i}\}.
$$
If $\Phi$ is the empty set, then $\beta+\varepsilon_{i\rho}$ is not
a root for any vertex $(i, \rho)\in\widehat{Q}$, and so that the
connected component of $\widehat{Q}$ which $\beta$ lies in is Dynkin
(see \cite[Proposition 4.9]{Ka}). Therefore, $\widehat{Q}$ must be a
disjoint union of copies of this Dynkin quiver, all in a single
$G$-orbit. Thus $\widehat{Q}$ and $Q$ are representation finite
\cite{RR}, $\Gamma$ is a connected Dynkin diagram. This contradicts
to that $\alpha$ is a imaginary root.

It follows that $\Phi$ is non-empty. We denote by $\widetilde{\Phi}$
the full subgraph of $\Gamma$ determined by $\Phi$. Let $T$ be a
non-empty connected component of $\Gamma-\widetilde{\Phi}$, and let
$\widetilde{\beta}$ be the restriction of $h(\beta)$ to $T$. If
$T\neq\emptyset$, then for all vertices $j\in T$, we have
$(\widetilde{\beta}, \overline{\varepsilon}_{j})_{T}\geq(h(\beta),
\overline{\varepsilon}_{j})_{\Gamma}\geq0$, where $(-,-)_{T}$ is the
restriction of $(-,-)_{\Gamma}$ on $T$. Moreover, note that there
exists a vertex $j\in T$ adjacent to $\widetilde{\Phi}$, we have
$(\widetilde{\beta}, \overline{\varepsilon}_{j})_{T}>0$. Therefore,
$T$ is a Dynkin diagram \cite[Corollary 4.9]{Ka}. On the other hand,
let $\widetilde{\beta}'$ be the restriction of $\alpha-h(\beta)$ to
$T$. Then $\widetilde{\beta}'$ has the support $T$, and for any
vertex $j\in T$,
$$
(\widetilde{\beta}',
\overline{\varepsilon}_{j})_{T}=(\alpha-h(\beta),
\overline{\varepsilon}_{j})_{\Gamma}=(\alpha,
\overline{\varepsilon}_{j})_{\Gamma}-(h(\beta),
\overline{\varepsilon}_{j})_{\Gamma}\leq 0.
$$
Hence $T$ is not
Dynkin. This is a contradiction. Therefore, $T$ is empty,
$\widetilde{\Phi}=\Gamma$ and so that $h(\beta)=\alpha$. Thus, we
have shown that $h$ is surjective by Lemma \ref{lem3-11}(3).

In general, assume that $\Gamma$ is non-connected. In this case,
$F_{\Gamma}=\bigcup F_{\Gamma'}$, where $\Gamma'$ run over all
connected components of $\Gamma$. By the discussion above, we see
that any element $\alpha\in F_{\Gamma}$, there exists an element
$\beta\in\Delta_{\widehat{Q}}$ such that $h(\beta)=\alpha$. Hence,
$h$ is also surjective.

Finally, for any real root $\alpha\in\Delta_{\Gamma}$, we let
$\beta\in\Delta_{\widehat{Q}}$ be the element such that
$h(\beta)=\alpha$. Then, there is an element $\omega'\in\mathcal
{W}(\Gamma)$ and $i\in\mathcal{I}$ such that
$\omega'(\alpha)=\overline{\varepsilon}_{i}$. Let $\omega$ be the
element in $C_{G}(\mathcal{W}(\widehat{Q}))$ corresponding to
$\omega'$. It follows that $\omega(\beta)$ must also be a simple
root $\varepsilon_{i\rho}$ for some $\rho\in\irr G_{i}$. Therefore
$\beta$ is real and uniquely determined up to a $G$-orbit.
\end{proof}

Consider the action of $G$ on $\mathbbm{k}\widehat{Q}$, any $g\in G$
also induces an additive autoequivalence functor $F_{g}:\mbox{{\bf
mod}-}\mathbbm{k}\widehat{Q} \rightarrow \mbox{{\bf
mod}-}\mathbbm{k}\widehat{Q}$, $X \mapsto {^{g}X}$. Here we also
denote by $G_{X}$ a complete set of left coset representatives of
$H_{X}:=\{g\in G\mid F_{g}(X)\cong X\}$ in $G$, for any
$X\in\mbox{{\bf mod}-}\mathbbm{k}\widehat{Q}$. Following from Kac
Theorem, for any positive real root $\beta\in\Delta_{\widehat{Q}}$,
there exists a unique $\widehat{Q}$-representation $X$ such that
${\bf dim}X=\beta$ and $H_{X}=H_{\beta}$. By Proposition
\ref{prop3-12}, there are $|G_{X}|$ indecomposable
$\widehat{Q}$-representations (up to isomorphism) such that the
image of their dimension vector under the map $h$ are $\alpha$, if $h({\bf
dim}X)=\alpha$. Thus the proof of Theorem \ref{thm1-1} is completed.


\section{Proof of Theorem \ref{thm1-2}} \label{sect-4}
Assume that $G\subseteq\Aut(\mathbbm{k}Q)$ is a finite abelian
group. In this section, we lift $G$ to
$\overline{G}\subseteq\Aut(\mathfrak{g})$ such that the Kac-Moody
algebra $\mathfrak{g}(\Gamma)$ can be embedded into the fixed point
algebra $\mathfrak{g}^{\overline{G}}$. In this case,
$\mathfrak{g}^{\overline{G}}$ is integrable  as a
$\mathfrak{g}(\Gamma)$-module.

\medskip
Firstly, we recall some notations of Kac-Moody algebras. For a
symmetricable generalized Cartan matrix $C=(c_{ij})$ of size $n$ and
rank $l$, there exist a diagonal matrix
$D=\mbox{diag}(d_{1},\cdots,d_{n})$ and a symmetric matrix
$B=(b_{ij})$  such that $C=D^{-1}B$. In fact, $d_i (1\leq i\leq n)$
may be chosen to be positive integers. Let $\mathfrak{h}$ be a
$2n-l$ dimension $\mathbbm{k}$-vector space. Choose linearly
independent sets
 $\left\{H_i \in\mathfrak{h} | 1\leq i\leq n\right\}$ and
$\left\{\varepsilon_{i} \in\mathfrak{h}^{\ast} | 1\leq i\leq
n\right\}$ such that $\varepsilon_{j}(H_{i})=c_{ij}$. Then the
triple $\left(\mathfrak{h}, \{\varepsilon_{i}\}, \{H_{i}\}\right
)_{1\leq i\leq n}$ is called a (minimal) realization of $C$. Since
any two realizations of $C$ are isomorphic, there is a unique (up to
isomorphism) Kac-Moody algebra $\mathfrak{g}(C)$ generated by
$\mathfrak{h}$, $E_{i}$, $F_{i}$, $1\leq i\leq n$, with relations
$$
\begin{array}{llll} \quad[H, H']=0,\qquad\qquad &[H,
E_{j}]=\varepsilon_{j}(H)E_{j}, \qquad & (\ad
E_{i})^{1-c_{ij}}E_{j}=0,\\
\quad[E_{i}, F_{j}]=\delta_{ij}H_{i}, & [H,
F_{j}]=-\varepsilon_{j}(H)F_{j}, & (\ad F_{i})^{1-c_{ij}}F_{j}=0.
\end{array}
$$
for any $H, H'\in\mathfrak{h}$, where $\delta_{ij}$ is the Kronecker
sign. Moreover, the center $\mathfrak{c}$ of $\mathfrak{g}(C)$ is
given by
$$
\{H\in\mathfrak{h}\mid\varepsilon_{i}(H)=0 \mbox{ for all } 1\leq
i\leq n\}\subseteq [\mathfrak{g}(C), \mathfrak{g}(C)].
$$
For the details one can see \cite{Ka}.

\medskip
For the  pair $(Q, G)$, we have obtained the valued graph $\Gamma$
with symmetricable generalized Cartan matrix $C=(c_{ij})$ of size
$|\mathcal {I}|$ and the generalized McKay quiver $\widehat{Q}$ with
symmetric generalized Cartan matrix
$\widehat{A}=(a_{(i\rho)(j\sigma)})$ of size $|\widehat{I}|$, see
Section \ref{sect-2}. Therefore we have Kac-Moody algebras
$\mathfrak{g}(\Gamma):=\mathfrak{g}(C)$ corresponding to the
realization $\big(\mathfrak{h}(\Gamma),
\{\overline{\varepsilon}_{i}\}, \{\overline{H}_{i}\}\big)$ of $C$
and
$\mathfrak{g}:=\mathfrak{g}(\widehat{Q})=\mathfrak{g}(\widehat{A})$
corresponding to the realization $\big(\mathfrak{h},
\{\varepsilon_{i\rho}\}, \{H_{i\rho}\}\big)$ of $\widehat{A}$.
Denote by $r$ and $s$ the coranks of $C$ and $\widehat{A}$, then
$\dim_{\mathbbm{k}}\mathfrak{h}(\Gamma)=|\mathcal {I}|+r$ and
$\dim_{\mathbbm{k}}\mathfrak{h}=|\widehat{I}|+s$.

We suppose that $\mathfrak{g}(\Gamma)$ generated by
$\mathfrak{h}(\Gamma)$ and $\overline{E}_{i}, \overline{F}_{i}$,
$i\in\mathcal{I}$.  There is a symmetric bilinear form
$(-,-)_{\Gamma}$ on $\mathfrak{h}(\Gamma)$ such that
$$
(\overline{H}_{i},\overline{H})_{\Gamma}=\frac{1}{d_{i}}
\overline{\varepsilon}_{i}(\overline{H})
$$
for all
$\overline{H}\in\mathfrak{h}(\Gamma)$. Then we can extend it
uniquely to an invariant non-degenerate symmetric bilinear form on
$\mathfrak{g}(\Gamma)$ such that
$$
(\overline{E}_{i}, \overline{F}_{i})_{\Gamma}=\frac{1}{d_{i}}.
$$
Moreover, $(-,-)_{\Gamma}$ determines a bijection $\nu :
\mathfrak{h}(\Gamma)\rightarrow\mathfrak{h}^{\ast}(\Gamma)$ sending
$\overline{H}_{i}$ to $\frac{1}{d_{i}}\overline{\varepsilon}_{i}$,
and hence induces a bilinear form on $\mathfrak{h}^{\ast}(\Gamma)$.
We also denote this bilinear form by $(-,-)_{\Gamma}$. Note that
$(\overline{\varepsilon}_{i},
\overline{\varepsilon}_{i})_{\Gamma}=b_{ij}$. It recovers the
bilinear form defined in Section \ref{sect-2}.3 for the root lattice
$\mathbb{Z}\mathcal{I}$. Similarly, there is a symmetric bilinear
form on $\mathfrak{h}^{\ast}=\mathfrak{h}^{\ast}(\what{Q})$ with
$(\varepsilon_{i\rho}, \varepsilon_{j\sigma})_{\widehat{Q}}
=a_{(i\rho)(j\sigma)}.$

\medskip
We now consider the action of $G$ on the quiver $\widehat{Q}$
defined in Section \ref{sect-3}.2. Recall that the derived algebra
$\mathfrak{g}'$ of $\mathfrak{g}$
 is generated by $H_{i\rho}$, $E_{i\rho}$,
$F_{i\rho}$, $(i, \rho)\in\widehat{I}$ and the action of $G$ on
$\widehat{Q}$ satisfies
$$
a_{(i\rho)(j\sigma)}=a_{(i\rho')(j\sigma')},
\quad\hbox{ if } (i, \rho')=g(i, \rho) \hbox{ and } (j,
\sigma')=g(j, \sigma)
$$
for some $g\in G$. Then, there is a natural action of $G$ on
$\mathfrak{g}'$ given by
$$
g(H_{i\rho})=H_{i\rho'},\quad g(E_{i\rho})=E_{i\rho'},
\quad g(F_{i\rho})=F_{i\rho'}
$$
for any $g\in G$. Denote by
$\mathfrak{h}'(\Gamma)$ and $\mathfrak{h}'$ the Cartan subalgebra of
$\mathfrak{g}'(\Gamma):=[\mathfrak{g}(\Gamma),
\mathfrak{g}(\Gamma)]$ and $\mathfrak{g}'$ respectively. It is easy
to see that the map
$$
\phi :\quad
\mathfrak{h}'(\Gamma)\rightarrow(\mathfrak{h}')^{G}
$$
given by
$\phi(\overline{H}_{i})=\sum_{\rho\in{\rm irr} G_{i}}H_{i\rho}$ is
an isomorphism and
$$
(\overline{H},
\overline{H}')_{\Gamma}=\frac{1}{|G|}(\phi(\overline{H}),
\phi(\overline{H}'))_{\widehat{Q}}
$$
for $\overline{H},
\overline{H}\in\mathfrak{h}'(\Gamma)$. In particular, the fixed
point subalgebra $\mathfrak{c}^{G}$ of the center of
$\mathfrak{g}(\widehat{Q})$ is isomorphic to the center
$\mathfrak{c}(\Gamma)$ of $\mathfrak{g}(\Gamma)$.

\medskip
We wish to extend the action of $G$ on $\mathfrak{g}'$ to the whole
Lie algebra $\mathfrak{g}$. Let $\Aut(\widehat{A})$ denote the set
of permutations $g$ of $\widehat{I}$ satisfying
$$
a_{(i\rho)(j\sigma)}=a_{(l\rho')(k\sigma')} \quad\hbox{ if  } (l,
\rho')=g(i, \rho)\hbox{ and } (k, \sigma')=g(j, \sigma).
$$
Let $\DAut(\mathfrak{g})$ denote the subgroup of
$\Aut(\mathfrak{g})$ consisting of the automorphisms preserving each
of the sets $\mathfrak{h}$, $\{E_{i\rho}\}$ and $\{F_{i\rho}\}$.

\begin{proposition} \label{prop4-1}
(see \cite[Section 4.19]{KW}) There is a short exact sequence
$$
0\rightarrow{\rm Hom}_{\mathbbm{k}} (\mathfrak{h}/\mathfrak{h}',
\mathfrak{c})\longrightarrow{\rm
DAut}(\mathfrak{g})\longrightarrow{\rm
Aut}(\widehat{A})\rightarrow0.
$$
\end{proposition}

\begin{proof} It is easy to see that
$\overline{g}(H_{i\rho})=H_{j\sigma}$,
$\overline{g}(E_{i\rho})=E_{j\sigma}$ and
$\overline{g}(F_{i\rho})=F_{j\sigma}$ for any
$\overline{g}\in\DAut(\mathfrak{g})$. Thus, there exists a unique
permutation $g\in\Aut(\widehat{A})$ corresponding to $\bar{g}$ such
that $(j, \sigma)=g(i, \rho)$. Moreover, each
$g\in\Aut(\widehat{A})$ can be obtained in this way.

Let $\Lambda:=\mathbbm{k}\widehat{I}$ be the subspace of
$\mathfrak{h}^{\ast}$ spanned by $\{\varepsilon_{i\rho}\mid
(i,\rho)\in\widehat{I}\}$. Then there is an natural action of
$\Aut(\widehat{A})$ on $\Lambda$:
$g(\varepsilon_{i\rho})=\varepsilon_{j\sigma}$, where $(j,
\sigma)=g(i, \rho)$, $g\in G$, and it induces an action of
$\Aut(\widehat{A})$ on the quotient space
$\mathfrak{h}/\mathfrak{c}$ since $\mathfrak{h}/\mathfrak{c}$ is
dual to $\Lambda$. It maps $H_{i\rho}\mod{\mathfrak{c}}$ to
$H_{j\sigma}\mod{\mathfrak{c}}$, and so that
$\mathfrak{h}'/\mathfrak{c}$ is $\Aut(\widehat{A})$-stable. Since
$\Aut(\widehat{A})$ is finite, there exists $\mathfrak{h}''$ such
that $\mathfrak{h}=\mathfrak{h}'\oplus\mathfrak{h}''$ and
$(\mathfrak{h}''+\mathfrak{c})/\mathfrak{c}$ is
$\Aut(\widehat{A})$-stable. For any $g\in\Aut(\widehat{A})$, we can
define an automorphism $\overline{g}\in\DAut(\mathfrak{g})$ by
$$
\overline{g}(H_{i\rho})=H_{j\sigma},\quad
\overline{g}(E_{i\rho})=E_{j\sigma} \quad\hbox{ and }\quad
\overline{g}(F_{i\rho})=F_{j\sigma},
$$
and $\overline{g}|_{\mathfrak{h}''}$ is the pull-back of $g$ on
$(\mathfrak{h}''+\mathfrak{c})/\mathfrak{c}$.

Obviously, the kernel of the map
$\DAut(\mathfrak{g})\rightarrow\Aut(\widehat{A})$ is the subgroup
$\Aut(\mathfrak{g}; \mathfrak{g}')$ consisting of all automorphisms
acting trivially on $\mathfrak{g}'$. One can check that an
automorphism $\alpha\in\Aut(\mathfrak{g}; \mathfrak{g}')$ if and
only if there exists a map $\varphi :
\mathfrak{h}''\rightarrow\mathfrak{c}$ such that
$\alpha(H)=H+\varphi(H)$ for all $H\in\mathfrak{h}''$. Thus, there
are isomorphisms $\Aut(\mathfrak{g};
\mathfrak{g}')\cong\Hom_{\mathbbm{k}}(\mathfrak{h}'',
\mathfrak{c})\cong\Hom_{\mathbbm{k}}(\mathfrak{h}/\mathfrak{h}',
\mathfrak{c})$.
\end{proof}

Therefore, for each $\alpha\in\Aut(\mathfrak{g}; \mathfrak{g}')$ and
$g\in\Aut(\widehat{A})$, we have an element
$\overline{g}\in\DAut(\mathfrak{g})$ by setting
$\overline{g}|_{\mathfrak{g}'}=g$ and
$\overline{g}|_{\mathfrak{h}''}=\alpha$. Moreover, for any
$\alpha\in\Aut(\mathfrak{g}; \mathfrak{g}')$ corresponding to
$\varphi : \mathfrak{h}''\rightarrow\mathfrak{c}$, it is easy to see
that $\alpha^{t}(H)=H+t\varphi(H)$ for any $t\in\mathbb{Z}$ and
$H\in\mathfrak{h}''$. That is to say, an automorphism
$\alpha\in\Aut(\mathfrak{g}; \mathfrak{g}')$ has finite order if and
only if the corresponding map $\varphi :
\mathfrak{h}''\rightarrow\mathfrak{c}$ is zero.

\medskip
We now fix $\Omega=\{g_{1}, g_{2},\cdots,g_{n}\}$ a set of
generators of $G$. We can view $G$ as a finite abelian subgroup of
$\Aut(\widehat{A})$. By Proposition \ref{prop4-1}, we can lift $G$
to an automorphism group $\overline{G}=\{\overline{g}\mid g\in G\}$
of $\mathfrak{g}$ corresponding to a set of linear maps
$\{\varphi_{i}=\varphi_{g_{i}} :
\mathfrak{h}''\rightarrow\mathfrak{c}\mid g_{i}\in\Omega\}$. It is
easy to see that for any $H\in\mathfrak{h}$, we have
$\varepsilon_{i\rho'}(\overline{g}(H))=\varepsilon_{i\rho}(H)$ if
$(i, \rho')=g(i, \rho)$. Let
$$
\mathcal{S}:=\mbox{span}\{\varepsilon_{i\rho}-\varepsilon_{i\rho'}
\mid i\in\mathcal{I}, ~\rho, \rho'\in \irr
G_{i}\}\subseteq\mathfrak{h}^{\ast}
$$
and
$$
\mathcal{H}:=\{H\in\mathfrak{h}\mid\varepsilon_{i\rho}(H)
=\varepsilon_{i\rho'}(H)\mbox{ for all } \rho, \rho'\in \irr G_{i},
\mbox{ and } i\in\mathcal{I}\}
=\mbox{ann}_{\mathfrak{h}}\mathcal{S}.
$$
Then $\mathcal{H}$ contains the center $\mathfrak{c}$,
$\mathcal{H}/\mathfrak{c}=(\mathfrak{h}/\mathfrak{c})^{G}$ and so
that, for any lifting $\overline{G}$ of $G$,
$\mathcal{H}^{\overline{G}}=\mathfrak{h}^{\overline{G}}.$

\begin{lemma} \label{lem4-2}
$\mathcal{H}$ has $\mathbbm{k}$-dimension $|\mathcal {I}|+s$,
~$\mathcal{H}\cap\mathfrak{h}'$ has $\mathbbm{k}$-dimension
$|\mathcal {I}|+s-r$ and therefore $\mathcal{H}\cap\mathfrak{h}''$
has $\mathbbm{k}$-dimension $r$.
\end{lemma}

\begin{proof}
Note that
$$
\{\varepsilon_{i\rho}-\varepsilon_{i\rho'}\mid i\in\mathcal{I},
~\rho'\in\irr G_{i}\setminus\rho\}
$$
is a basis of $\mathcal{S}$, we obtain that
$\dim_{\mathbbm{k}}\mathcal{H}=\dim_{\mathbbm{k}}\mathfrak{h}
-\dim_{\mathbbm{k}}\mathcal{S}=|\mathcal{I}|+s$. Since
$(\mathcal{H}\cap\mathfrak{h}')/\mathfrak{c}
=(\mathfrak{h}'/\mathfrak{c})^{G}$ is isomorphic to
$(\mathfrak{h}')^{G}/\mathfrak{c}^{G}$,
~$\dim_{\mathbbm{k}}(\mathfrak{h}')^{G}=|\mathcal{I}|$ and
$\dim_{\mathbbm{k}}\mathfrak{c}^{G}
=\dim_{\mathbbm{k}}\mathfrak{c}(\Gamma)=r$,
$\mathcal{H}\cap\mathfrak{h}'$ has $\mathbbm{k}$-dimension
$|\mathcal {I}|+s-r$ and so that $\mathcal{H}\cap\mathfrak{h}''$ has
$\mathbbm{k}$-dimension $r$.
\end{proof}

\begin{proposition} \label{prop4-3}
Let $\overline{G}$ be a lifting of $G$
to $\mathfrak{g}$ corresponding to $\{\varphi_{i} :
\mathfrak{h}''\rightarrow\mathfrak{c}\mid 1\leq i\leq n\}$. Then
$$
\bigg(\mathcal{H}^{\overline{G}},
\bigg\{\frac{d_{i}}{|G|}\sum_{\rho\in {\rm
irr}G_{i}}\varepsilon_{i\rho}\bigg\}, \bigg\{\sum_{\rho\in {\rm
irr}G_{i}}H_{i\rho}\bigg\}\bigg)
$$
is a realization of $C$ if and
only if $\varphi_{i}(\mathcal{H}\cap\mathfrak{h}'')=0$ for all
$1\leq i\leq n$.
\end{proposition}

\begin{proof} We denote by
$$
H_{i}:=\sum_{\rho\in {\rm
irr}G_{i}}H_{i\rho} \quad\hbox{ and }\quad
\epsilon_{i}:=\frac{d_{i}}{|G|}\sum_{\rho\in {\rm
irr}G_{i}}\varepsilon_{i\rho}
$$
for all $i\in\mathcal{I}$. Since
$\{H_{i}\mid i\in\mathcal{I}\}$ is a basis of
$(\mathcal{H}\cap\mathfrak{h}')^{G}$, $\mathcal{H}^{\overline{G}}$
has dimension $|\mathcal{I}|+r$ if and only if there are $h'_{1},
h'_{2}, \cdots, h'_{r}\in\mathcal{H}^{\overline{G}}$ spanning the
complementary space of $(\mathcal{H}\cap\mathfrak{h}')^{G}$ in
$\mathcal{H}^{\overline{G}}$.

Since $(\mathfrak{h}''+\mathfrak{c})/\mathfrak{c}$ is $G$-stable,
 $\big((\mathfrak{h}''+\mathfrak{c})/\mathfrak{c}\big)^{G}$
has $\mathbbm{k}$-dimension $r$ by Lemma \ref{lem4-2}. We can find
linearly independent elements $h''_{1}, h''_{2}, \cdots, h''_{r}
\in\mathcal{H}\cap\mathfrak{h}''$ such that $h''_{i}\mod{
\mathfrak{c}}$ are fixed by $G$. Since
$\varphi_{i}(\mathcal{H}\cap\mathfrak{h}'')=0$ for all $i$,
$h''_{1}, h''_{2}, \cdots, h''_{r}$ are $G$-stable and form a basis
of $\mathcal{H}\cap\mathfrak{h}''$. Therefore, we take
$h'_{i}=h''_{i}$ for all $1\leq i\leq r$. On the other hand, if we
can find such elements $h'_{1}, h'_{2}, \cdots, h'_{r}$, then each
$h''_{i}$ has the form
$$
h''_{i}=\sum_{j=1}^{s}p_{ij}h'_{j}-\sum_{(j,
\sigma)\in\widehat{I}} q_{i(j\sigma)}H_{j\sigma}
$$
for some $p_{ij}, q_{i(j\sigma)}\in\mathbbm{k}$, and
$$
\begin{aligned}\varphi_{l}(h''_{i})
&=\overline{g}_{l}\Big(\sum_{j=1}^{s}p_{ij}h'_{j} -\sum_{(j,
\sigma)\in\widehat{I}} q_{i(j\sigma)}H_{j\sigma}\Big)
-\sum_{j=1}^{s}p_{ij}h'_{j}+\sum_{(j, \sigma)\in\widehat{I}}
q_{i(j\sigma)}H_{j\sigma}\nonumber\\
&=\sum_{(j, \sigma)\in\widehat{I}}
q_{i(j\sigma)}(H_{j\sigma}-H_{j\sigma^{1}}),
\end{aligned}
$$
where $(j, \sigma^{1})=g_{l}(j, \sigma)$. It follows that
$$
t\varphi_{l}(h''_{i})=\sum_{(j, \sigma)\in\widehat{I}}
q_{i(j\sigma)}(H_{j\sigma}-H_{j\sigma^{t}})
$$
for any $t\in\mathbb{Z}$, where $(j, \sigma^{t})=g_{l}^{t}(j,
\sigma)$. Note that $\widehat{I}$ is a finite set, there exist some
$t\in\mathbb{Z}$ such that $g_l^{t}(j, \sigma)=(j, \sigma)$ for all
$(j, \sigma)\in\widehat{I}$, and so that $t\varphi_{l}(h''_{i})=0$,
$\varphi_{l}(h''_{i})=0$ for all $i$ and $l$. Thus
$\varphi_{i}(\mathcal{H}\cap\mathfrak{h}'')=0$ for any $1\leq i\leq
n$.

Since
$$
\epsilon_{j}(H_{i})=\frac{d_{i}}{|G|} \sum_{\rho\in{\rm irr}
G_{i}\atop\sigma\in{\rm irr}
G_{j}}\varepsilon_{j\sigma}(H_{i\rho})=\frac{d_{i}}{|G|}
\sum_{\rho\in{\rm irr} G_{i}\atop\sigma\in{\rm irr}
G_{j}}a_{(i\rho)(j\sigma)}=c_{ij}
$$
and $H_{i}$ $(i\in\mathcal{I})$ are linearly independent, it remains
to show $\epsilon_{i}$, $i\in\mathcal{I}$ are linearly independent
modulo $\mbox{ann}_{\mathfrak{h}^{\ast}}({\mathcal
{H}}^{\overline{G}})$. Let
$$
\epsilon:=\sum_{j\in\mathcal{I}}\mu_{j}\epsilon_{j}\in
\mbox{ann}_{\mathfrak{h}^{\ast}}({\mathcal {H}}^{\overline{G}}),
\quad \mu_{j}\in\mathbbm{k}.
$$
Then
$$
0=\epsilon(H_{i})=\sum_{j\in\mathcal{I}}\mu_{j}\epsilon_{j}(H_{i})
=\sum_{j\in\mathcal{I}}c_{ij}\mu_{j}
$$
for all $i\in\mathcal{I}$, and so that
$$
\epsilon(H_{i\rho})=\sum_{j\in\mathcal{I}}\mu_{j}\epsilon_{j}(H_{i\rho})
=\frac{1}{|G|}\sum_{j\in\mathcal{I}}b_{ij}\mu_{j}=\frac{d_{i}}{|G|}
\sum_{j\in\mathcal{I}}c_{ij}\mu_{j}=0
$$
for all $(i, \rho)\in\widehat{I}$. Therefore,
$$
\epsilon\in\mbox{ann}_{\mathfrak{h}^{\ast}} ({\mathcal
H}^{\overline{G}}+\mathfrak{h}')
=\mbox{ann}_{\mathfrak{h}^{\ast}}({\mathcal
H}+\mathfrak{h}')\subseteq
\mbox{ann}_{\mathfrak{h}^{\ast}}({\mathcal H})=\mathcal{S},
$$
and
$$
g(\epsilon)=g\bigg(\sum_{j\in\mathcal{I}}\mu_{j}\epsilon_{j}\bigg)
=g\bigg(\sum_{j\in\mathcal{I}}\frac{d_{j}\mu_{j}}{|G|}
\sum_{\sigma\in{\rm irr}
G_{j}}\varepsilon_{j\sigma}\bigg)=\epsilon
$$
for any $g\in G$. It concludes that
$\epsilon=\sum_{j\in\mathcal{I}}\mu_{j}\epsilon_{j}=0$ by the
construction of $\mathcal{S}$. Therefore, $\mu_{j}=0$ for all
$j\in\mathcal{I}$, and so that $\epsilon_{j}$ are linearly
independent in $\mathfrak{h}^{\ast}$.

The proof is completed.
\end{proof}

\begin{remark} \label{rem4-4}
Since
$$
\Hom_{\mathbbm{k}}(\mathfrak{h}'',
\mathfrak{c})\cong\Hom_{\mathbbm{k}}(\mathfrak{h}/\mathfrak{h}',
\mathfrak{c}),
$$
for any lifting $\overline{G}$ of $G$, there exist a family of maps
$\{\psi_{i}=\psi_{g_{i}} :
\mathfrak{h}/\mathfrak{h}'\rightarrow\mathfrak{c}\mid
g_{i}\in\Omega\}$ corresponding to it. Moreover, it is easy to see
that the condition $\varphi_{i}(\mathcal{H}\cap\mathfrak{h}'')=0$ is
equivalent to
$\psi_{i}((\mathcal{H}+\mathfrak{h}')/\mathfrak{h}')=0$.
\end{remark}

Now we can prove the main results of this section.

\begin{proposition} \label{prop4-5}
There is a monomorphism
$\mathfrak{g}'(\Gamma)\rightarrow(\mathfrak{g}')^{G}$, and for the
lifting $\overline{G}$ of $G$ corresponding to $\{\varphi_{i} :
\mathfrak{h}''\rightarrow\mathfrak{c}\mid 1\leq i\leq n\}$ with
$\varphi_{i}(\mathcal{H}\cap\mathfrak{h}'')=0$, we can extend this
monomorphism to the whole Lie algebra such that
$$
\mathfrak{g}(\Gamma)\rightarrow\mathfrak{g}^{\overline{G}}
$$
is also a monomorphism.
\end{proposition}

\begin{proof} We set
$$
H_{i}:=\sum_{\rho\in{\rm irr} G_{i}}H_{i\rho}, \quad
E_{i}:=\sum_{\rho\in{\rm irr} G_{i}}E_{i\rho}, \quad
F_{i}:=\sum_{\rho\in{\rm irr} G_{i}}F_{i\rho}
$$
for all $i\in\mathcal{I}$. Then $H_{i}, E_{i},
F_{i}\in(\mathfrak{g}')^{G}$ and
\begin{eqnarray*}
&&[H_{i}, H_{j}]=0, \\
&&[E_{i}, F_{j}]=\sum_{\rho\in{\rm irr} G_{i}\atop \sigma\in{\rm
irr} G_{j}}[E_{i\rho}, F_{j\sigma}]=\delta_{ij}\sum_{\rho\in{\rm
irr}
G_{i}}H_{i\rho}=\delta_{ij}H_{i}, \\
&&[H_{i}, E_{j}]=\sum_{\rho\in{\rm irr} G_{i}\atop\sigma\in{\rm irr}
G_{j}}[H_{i\rho}, E_{j\sigma}]=\sum_{\rho\in{\rm irr}
G_{i}\atop\sigma\in{\rm irr}
G_{j}}a_{(i\rho)(j\sigma)}E_{j\sigma}=c_{ij}\sum_{\sigma\in{\rm irr}
G_{j}}E_{j\sigma}=c_{ij}E_{j}.
\end{eqnarray*}
Similarly, we have $[H_{i}, F_{j}]=c_{ij}F_{j}$ for any $i,
j\in\mathcal{I}$. Note that $\ad E_{i\rho}$ and $\ad E_{i\rho'}$
commute for any $\rho, \rho'\in\irr G_{i}$, we have
$$
(\ad
E_{i})^{n}=\sum_{\lambda}\Phi_{\lambda}^{n}\prod_{\rho\in{\rm irr}
G_{i}}(\ad E_{i\rho})^{\lambda_{\rho}}
$$
for any positive integer
$n$, where $\lambda$ takes thought all the sequence
$\lambda=(\lambda_{\rho})_{\rho\in{\rm irr} G_{i}}$ satisfying
$$
\sum_{\rho\in{\rm irr} G_{i}}\lambda_{\rho}=n,
$$
and
$$
\Phi_{\lambda}^{n}=\left( {\begin{array}{*{20}c} n \\
\rho_{1} \\ \end{array}} \right) \left({\begin{array}{*{20}c}
n-\rho_{1} \\ \rho_{2} \\ \end{array}} \right)\cdots\left(
{\begin{array}{*{20}c} n-\rho_{1}
-\cdots-\rho_{|{\rm irr} G_{i}|-1} \\ \rho_{|{\rm irr} G_{i}|} \\
\end{array}} \right)
$$
for any $\lambda=(\rho_{1}, \rho_{2},\cdots, \rho_{|{\rm irr}
G_{i}|})$. In particular, if $n=1-c_{ij}$, then
$\lambda_{\rho}>1-a_{(i\rho)(j\sigma)}$ for some $\rho\in\irr G_{i}$
and so that
$$
(\ad E_{i\rho})^{\lambda_{\rho}}E_{j\sigma}=0, \qquad (\ad
E_{i})^{1-c_{ij}}E_{j}=0.
$$
Similarly, $(\ad
F_{i})^{1-c_{ij}}F_{j}=0$ for any $i, j\in\mathcal{I}$. Therefore,
there exists a non-zero homomorphism
$\mathfrak{g}'(\Gamma)\rightarrow(\mathfrak{g}')^{G}$.

Since $\big(\mathcal{H}^{\overline{G}},
\{\frac{d_{i}}{|G|}\phi(\varepsilon_{i})\}, \{\phi(H_{i})\}\big)$ is
a realization of $C$ by Proposition \ref{prop4-3}, there is an
isomorphism
$\mathfrak{h}(\Gamma)\rightarrow\mathcal{H}^{\overline{G}}$,
$\overline{H}_{i}\rightarrow H_{i}$. Therefore we can get a
homomorphism
$\mathfrak{g}(\Gamma)\rightarrow\mathfrak{g}^{\overline{G}}$ by
compositing the homomorphisms
$\mathfrak{g}'(\Gamma)\rightarrow(\mathfrak{g}')^{G}$ and
$\mathfrak{h}(\Gamma)\rightarrow\mathcal{H}^{\overline{G}}$. By
\cite[Proposition 1.7(b)]{Ka},
$\mathfrak{g}(\Gamma)\rightarrow\mathfrak{g}^{\overline{G}}$ and
$\mathfrak{g}'(\Gamma)\rightarrow(\mathfrak{g}')^{G}$ are
monomorphisms.
\end{proof}

Now, we can identify $\mathfrak{g}(\Gamma)$ with a subalgebra of
$\mathfrak{g}^{\overline{G}}$. Following from Section
\ref{sect-3}.1, the map
$$
h: \quad \mathbb{Z}\widehat{I}\rightarrow\mathbb{Z}\mathcal{I},
\qquad \beta\mapsto h(\beta), \quad h(\beta)_{i}=\sum_{\rho\in{\rm
irr} G_{i}}\beta_{i\rho},
$$
satisfies
$$
d_{i}\bigg(\beta,
\sum_{\rho\in{\rm irr}
G_{i}}\varepsilon_{i\rho}\bigg)_{\widehat{Q}}=(h(\beta),
\overline{\varepsilon}_{i})_{\Gamma}
$$
for all $\beta=\sum_{(i,
\rho)\in\widehat{I}}\beta_{i\rho}\varepsilon_{i\rho}\in\mathbb{Z}\widehat{I}$
and $h(\Delta_{\widehat{Q}})=\Delta_{\Gamma}$ by Proposition
\ref{prop3-12}.

\begin{proposition} \label{prop4-6}
The monomorphism
$\mathfrak{g}(\Gamma)\rightarrow\mathfrak{g}^{\overline{G}}$ endows
$\mathfrak{g}^{\overline{G}}$ with an integrable
$\mathfrak{g}(\Gamma)$-module structure under the adjoint action of
$\mathfrak{g}(\Gamma)$.
\end{proposition}

\begin{proof}
Firstly, we identity the realization $\big(\mathfrak{h}(\Gamma),
\{\overline{\varepsilon}_{i}\}, \{\overline{H}_{i}\}\big)$ with
$\big(\mathcal{H}^{\overline{G}}, \{\epsilon_{i}\}, \{H_{i}\}\big)$.
For any non-zero $\beta=\sum_{(i,
\rho)\in\widehat{I}}\beta_{i\rho}\varepsilon_{i\rho}\in\Delta_{\widehat{Q}}$
and $H\in\mathcal{H}^{\overline{G}}$, we have
$$
\varepsilon_{i\rho}(H)=\frac{d_{i}}{|G|}\sum_{\rho\in{\rm irr}
G_{i}}\varepsilon_{i\rho}(H)=\overline{\varepsilon}_{i}(H)
$$
and
$$
\beta(H)=\sum_{(i,
\rho)\in\widehat{I}}\beta_{i\rho}\varepsilon_{i\rho}(H)
=\sum_{i\in\mathcal{I}}\Big(\sum_{\rho\in{\rm irr}
G_{i}}\beta_{i\rho}\Big)\overline{\varepsilon}_{i}(H)
=\sum_{i\in\mathcal{I}}h(\beta)_{i}\overline{\varepsilon}_{i}(H)
=h(\beta)(H).
$$
Denote by $H_{\beta}=\{g\in G \mid g(\beta)=\beta\}$
and $G_{\beta}$ a complete set of left coset representatives of
$H_{\beta}$ in $G$. Then $H_{\beta}$ acts on the root space
$\mathfrak{g}_{\beta}$. Suppose that $x\in\mathfrak{g}_{\beta}$
satisfies $g(x)=x$ for any $g\in H_{\beta}$. Let
$$
\Sigma(x):=\sum_{g\in G_{\beta}}g(x).
$$
It is easy to see that $\Sigma(x)\in\mathfrak{g}^{\overline{G}}$ and
$$
[H, \Sigma(x)]=\sum_{g\in G_{\beta}}[H, g(x)]
=\sum_{g\in G_{\beta}}g(\beta)(H)g(x)=h(\beta)(H)\sum_{g\in
G_{\beta}}g(x)=h(\beta)(H)\Sigma(x)
$$
for all $H\in\mathcal{H}^{\overline{G}}$, since
$h(g(\beta))=h(\beta)$ for any $g\in G$. It follows that $\Sigma(x)$
lies in the weight space $(\mathfrak{g}^{\overline{G}})_{h(\beta)}$.
Note that each element in $\mathfrak{g}^{\overline{G}}$ can be
written as a sum of some $\Sigma(x)$ with
$x\in\mathfrak{g}_{\beta}$, $\beta\in\Delta_{\widehat{Q}}$, we
obtain that $\mathfrak{g}^{\overline{G}}$ is
$\mathfrak{h}(\Gamma)$-diagonalisable.

Secondly, it is easy to see that the non-zero weights of
$\mathfrak{g}^{\overline{G}}$ must be roots of $\Gamma$ since
$h(\Delta_{\widehat{Q}})=\Delta_{\Gamma}$. On the other hand,  every
root of $\Gamma$ is also a weight of $\mathfrak{g}^{\overline{G}}$
under the adjoint action by the monomorphism
$\mathfrak{g}(\Gamma)\rightarrow\mathfrak{g}^{\overline{G}}$.

Finally, for any $\beta\in\Delta_{\Gamma}$, the set
$\{\beta+k\overline{\varepsilon}_{i}\mid
k\in\mathbb{Z}\}\cap\Delta_{\Gamma}$ is finite. Thus the action of
$\overline{E}_{i}$ and $\overline{F}_{i}$ are local nilpotent on
$\mathfrak{g}^{\overline{G}}$. The proof is completed.
\end{proof}

Following from the proof of Proposition \ref{prop4-6},
$(\mathfrak{g}^{\overline{G}})_{\alpha}$ is spanned by the elements
$\Sigma(x)=\sum_{g\in G_{\beta}}g(x)$, where
$x\in\mathfrak{g}_{\beta}$ satisfies $g(x)=x$ for any $g\in
H_{\beta}$,  and $\beta\in\Delta_{\widehat{Q}}$ satisfies
$h(\beta)=\alpha$. Thus, by the action of $G$ on $\{E_{i\rho}\}$ and
$\{F_{i\rho}\}$,  $\ad H_{\beta}$ acts on $\mathfrak{g}_{\beta}$ is
identity and so that
$$
\dim_{\mathbbm{k}}(\mathfrak{g}^{\overline{G}})_{h(\beta)}=1
$$
for any simple root $\beta$. That is,
$\dim_{\mathbbm{k}}(\mathfrak{g}^{\overline{G}})_{\alpha}=1$ for all
simple root $\alpha\in\Delta_{\Gamma}$. Moreover, we  have the
following claim.

\begin{cliam} \label{claim4-7} ${\rm
dim}_{\mathbbm{k}}(\mathfrak{g}^{\overline{G}})_{\alpha}=1$ for any
real root $\alpha\in\Delta_{\Gamma}$.\end{cliam}

\begin{proof} We consider the automorphism
$$
\overline{r}_{i\rho}:=\exp(\ad
F_{i\rho})\exp(-\ad E_{i\rho})\exp(\ad F_{i\rho})
$$
of $\mathfrak{g}$. Then
$\overline{r}_{i\rho}(\mathfrak{g}_{\beta})=\mathfrak{g}_{r_{i\rho}(\beta)}$
and $\overline{r}_{i\rho}(H)=H-\varepsilon_{i\rho}(H)H_{i\rho}$ for
any $H\in\mathfrak{h}$ (see \cite[Lemma 3.8]{Ka}). Note that
$\overline{r}_{i\rho}$ and $\overline{r}_{i\rho'}$ commute for any
$\rho, \rho'\in\irr G_{i}$, we let
$$
\overline{S}_{i}:=\prod_{\rho\in{\rm
irr}G_{i}}\overline{r}_{i\rho}.
$$
Then, for any $H\in\mathcal{H}^{\overline{G}}$, we have
$$
\overline{S}_{i}(H)=H-\sum_{\rho\in{\rm
irr}G_{i}}\varepsilon_{i\rho}(H)H_{i\rho}=H
-\epsilon_{i}(H)\sum_{\rho\in{\rm irr}G_{i}}H_{i\rho}=H
-\epsilon_{i}(H)H_{i},
$$
and
$\overline{S}_{i}(H)\in\mathcal{H}^{\overline{G}}$. Note that
$\overline{S}_{i}$ and $G$ commute on $\mathfrak{g}'$, it deduces
that $\overline{S}_{i}$ can define an automorphism of
$\mathfrak{g}^{\overline{G}}$ such that
$$
\overline{S}_{i}\big((\mathfrak{g}^{\overline{G}})_{\alpha}\big)
=(\mathfrak{g}^{\overline{G}})_{\widehat{S}_{i}(\alpha)}.
$$
Therefore, $\overline{S}_{i}$ is an extension of the automorphism
$\exp(\ad\overline{F}_{i})\exp(-\ad\overline{E}_{i})
\exp(\ad\overline{F}_{i})$ of $\mathfrak{g}(\Gamma)$.

Let $\alpha\in\Delta_{\Gamma}$ be a real root. By Lemma
\ref{lem3-11} and Proposition \ref{prop3-12}, there exist a real
root $\beta\in\Delta_{\widehat{Q}}$ and $\omega\in
C_{G}(\mathcal{W}(\widehat{Q}))$ such that $h(\beta)=\alpha$,
$\omega(\beta)$ is a simple root and $H_{\omega(\beta)}=H_{\beta}$.
Let
$\omega=\widehat{S}_{i_{1}}\widehat{S}_{i_{2}}\cdots\widehat{S}_{i_{r}}$
and $\overline{\omega}=\overline{S}_{i_{1}}\overline{S}_{i_{2}}
\cdots\overline{S}_{i_{r}}$, then
$\overline{\omega}(\mathfrak{g}_{\beta})=\mathfrak{g}_{\omega(\beta)}$
and hence $\mathfrak{g}_{\beta}$ is fixed by $H_{\beta}$. Finally,
note that all these $\beta$ are in the same $G$-orbit, we have
$\dim_{\mathbbm{k}}(\mathfrak{g}^{\overline{G}})_{\alpha}=1$.
\end{proof}

In particular, if $Q$ is a finite union of Dynkin quivers, then
$\mathfrak{g}$ is a direct sum of simple Lie algebras and all roots
of $\Gamma$ are real. By the claim, we have

\begin{corollary} \label{cor4-8}
If $Q$ is a finite union of Dynkin quivers and $G\subseteq{\rm
Aut}(\mathbbm{k}Q)$ is finite abelian, then there is a Lie algebra
isomorphism $\mathfrak{g}(\Gamma)\cong\mathfrak{g}^{\overline{G}}$.
\end{corollary}


\section{Examples} \label{sect-5}
In this section, we give two examples to elucidate our results.

\begin{example} \label{exa5-1}
Let $Q=(I, E)$ be the quiver
$$~~~~~~~~~~~~~~~~~~~~~~~\setlength{\unitlength}{1mm}
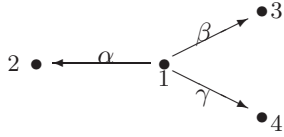
\begin{picture}(50,18)
\put(22,8){$\bullet$}             \put(35,1){$\bullet$}
\put(35,15){$\bullet$}            \put(5,8){$\bullet$}
\put(37,15){\small $3$}         \put(37,0){\small $4$}
\put(22,6){\small $1$}           \put(2,8){\small $2$}
\put(24,8){\vector (2,-1){10}}   \put(24,10){\vector (2,1){10}}
\put(21,9){\vector (-1,0){13}}   \put(14,9){$\alpha$}
\put(27,12){$\beta$}             \put(27,4){$\gamma$}
\end{picture}
$$
The action of $G=\langle g\rangle\cong\mathbb{Z}/6\mathbb{Z}$ on
$\mathbbm{k}Q$ given by
\begin{center}
\begin{tabular*}{6cm}{@{\extracolsep{\fill}}l|lllllllllr}
& $e_{1}$ & $e_{2}$ & $e_{3}$ & $e_{4}$
 & $\alpha$ & $\beta$ & $\gamma$ \\
\hline $g$ & $e_{1}$ & $e_{3}$ & $e_{4}$ & $e_{2}$ & $-\beta$ &
$-\gamma$ & $-\alpha$\\
\end{tabular*}
\end{center}
where $e_{i}$ is the idempotent element of $\mathbbm{k}Q$
corresponding to vertex $i$, $i\in\{1, 2, 3, 4\}$. Then the Cartan
matrix of $Q$ is
$$
A=(a_{ij})={\small\left(
\begin{array}{cccc}2&-1&-1&-1\\
-1&2&0&0\\ -1&0&2&0\\  -1&0&0&2\\
\end{array}\right)}.
$$
Let $\varepsilon_{1}, \varepsilon_{2},
\varepsilon_{3}, \varepsilon_{4}$ be all the simple roots of the
symmetric Kac-Moody algebra $\mathfrak{g}(Q)$.  We endow the root
lattice $\mathbb{Z}I$ with a symmetric bilinear form $(-,-)_{Q}$ via
$(\varepsilon_{i}, \varepsilon_{j})_{Q}=a_{ij}$ and define
reflection $r_{i} :
\alpha\mapsto\alpha-(\alpha,\varepsilon_{i})_{Q}\varepsilon_{i}$ for
each vertex $i\in I$. Then, it is well-knowen that Weyl group
$\mathcal{W}(Q)\cong(\mathbb{Z}/2\mathbb{Z})^{3}\rtimes S_{4}$, and
one can check that $\Delta_{Q}=\pm\{\varepsilon_{1},
\varepsilon_{2}, \varepsilon_{3}, \varepsilon_{4},
\varepsilon_{1}+\varepsilon_{2}, \varepsilon_{1}+\varepsilon_{3},
\varepsilon_{1}+\varepsilon_{4},
\varepsilon_{1}+\varepsilon_{2}+\varepsilon_{3},
\varepsilon_{1}+\varepsilon_{2}+\varepsilon_{4},
\varepsilon_{1}+\varepsilon_{3}+\varepsilon_{4},
\varepsilon_{1}+\varepsilon_{2}+\varepsilon_{3}+\varepsilon_{4},
2\varepsilon_{1}+\varepsilon_{2}+\varepsilon_{3}+\varepsilon_{4}\}$
is the root system of $\mathfrak{g}(Q)$.

\medskip
We get the generalized McKay quiver $\widehat{Q}=(\widehat{I},
\widehat{E})$ of $(Q, G)$ as follows.
$$~~~~~~~~~~~~~~~~~~~~~~~\setlength{\unitlength}{1mm}
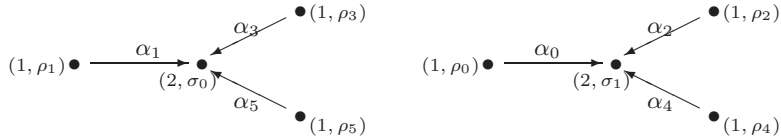
\begin{picture}(50,18)
\put(-3,8){$\bullet$}                    \put(10,1){$\bullet$}
\put(10,15){$\bullet$}                   \put(-20,8){$\bullet$}
\put(12,15){\scriptsize $(1,\rho_{3})$}
\put(12,0){\scriptsize$(1,\rho_{5})$}
\put(-8,6){\scriptsize$(2,\sigma_{0})$}
\put(-28,8){\scriptsize$(1,\rho_{1})$} \put(9,3){\vector(-2,1){10}}
\put(9,15){\vector (-2,-1){10}}  \put(-17,9){\vector (1,0){13}}
\put(-11,10){\small$\alpha_{1}$}      \put(2,13){\small$\alpha_{3}$}
\put(2,3){\small$\alpha_{5}$}           \put(52,8){$\bullet$}
\put(65,1){$\bullet$}                    \put(65,15){$\bullet$}
\put(35,8){$\bullet$}        \put(67,15){\scriptsize $(1,\rho_{2})$}
\put(67,0){\scriptsize$(1,\rho_{4})$} \put(47,6){\scriptsize
$(2,\sigma_{1})$}             \put(27,8){\scriptsize $(1,\rho_{0})$}
\put(64,3){\vector (-2,1){10}}    \put(64,15){\vector (-2,-1){10}}
\put(38,9){\vector (1,0){13}}     \put(42,10){\small$\alpha_{0}$}
\put(57,13){\small$\alpha_{2}$}    \put(57,3){\small$\alpha_{4}$}
\end{picture}
$$
where $\rho_{i}$ is the irreducible representation
of $G=\langle g\rangle\cong\mathbb{Z}/6\mathbb{Z}$ defined by
$$
a\cdot g=\xi^{i}a, \qquad a\in\rho_{i},
$$
$\sigma_{j}$ is the irreducible representation of $\langle
g^{3}\rangle\cong\mathbb{Z}/2\mathbb{Z}$ defined by
$$
b\cdot g^{3}=\xi^{3j}b, \qquad b\in\sigma_{j},
$$
and $\xi$ is a primitive $6$-th root of unity. As we have discussed
in Section \ref{sect-3}.2, by the group isomorphism
$$
\varphi: G\rightarrow\widetilde{G}, \qquad
\varphi(g^{i})=\chi_{g^{i}}, \qquad \chi_{g^{i}}(g^{j})=\xi^{ij},
$$
we define the action of $G$ on $\mathbbm{k}Q\ast G$ by setting
$$
g^{i}(\lambda g^{j})=\xi^{ij}\lambda g^{j}
$$
for any $g^{i}\in G$, $\lambda g^{j}\in\mathbbm{k}Q\ast G$. This
induces an action of $G=\langle g\rangle\cong\mathbb{Z}/6\mathbb{Z}$
on $\mathbbm{k}\widehat{Q}$ given by
\begin{center}
\begin{tabular*}
{12cm}{@{\extracolsep{\fill}}l|llllllllllllllllr} & $e_{0}$ &
$e_{1}$ & $e_{2}$ & $e_{3}$ & $e_{4}$ & $e_{5}$ & $e'_{0}$ &
$e'_{1}$
 & $\alpha_{0}$ & $\alpha_{1}$ & $\alpha_{2}$
 & $\alpha_{3}$ & $\alpha_{4}$ & $\alpha_{5}$  \\
\hline $g$ & $e_{1}$ & $e_{2}$ & $e_{3}$ & $e_{4}$ & $e_{5}$ &
$e_{0}$ & $e'_{1}$ & $e'_{0}$ & $\xi_{0}\alpha_{1}$ &
$\xi_{1}\alpha_{2}$ & $\xi_{2}\alpha_{3}$ & $\xi_{3}\alpha_{4}$ &
$\xi_{4}\alpha_{5}$ & $\xi_{5}\alpha_{0}$ \\
\end{tabular*}
\end{center}
where  idempotent elements $e_{i}$, $e'_{i}$ are corresponding to
the vertex $(1, \rho_{i})$,  $(2, \sigma_{i})$ respectively, and
$\xi_{i}\in\mathbbm{k}$ satisfying $\xi_{0}\xi_{1}\cdots\xi_{5}=1$.
One can check that the generalized McKay quiver of $(\widehat{Q},
G)$ is just the quiver $Q$.

\medskip
By the definition given in Section \ref{sect-2}.3, we obtain the
symmetrisable generalized Cartan matrix $C$ corresponding to $(Q,
G)$, i.e.,
$$C={\small\left(
\begin{array}{cc}2&-1\\-3&2\\\end{array}\right)}.
$$
Then the valued graph $\Gamma$ corresponding to $C$ is
$$~~~~~~~~~~~~~~~~~~~~~~~\setlength{\unitlength}{1mm}
\begin{picture}(6,4)
\put(-4,0){$\bullet$}             \put(14,0){$\bullet$}
\put(-18,0){$\Gamma :$}        \put(-2,1){\line (1,0){15}}
\put(1,2){\small $(3,~1)$}     \put(-6,0){\small $1$}
\put(16,0){\small $2$}
\end{picture}
$$
Let $\overline{\varepsilon}_{1}, \overline{\varepsilon}_{2}$ be all
the simple roots of $\Gamma$. Then the Weyl group
$$
\mathcal{W}(\Gamma)\cong D_{6}=\langle a, b\mid a^{2}=1, b^{3}=1,
ab=b^{-1}a\rangle
$$
and root system $\Delta_{\Gamma}=\{\overline{\varepsilon}_{1},
\overline{\varepsilon}_{2},
\overline{\varepsilon}_{1}+\overline{\varepsilon}_{2},
2\overline{\varepsilon}_{1}+\overline{\varepsilon}_{2},
3\overline{\varepsilon}_{1}+\overline{\varepsilon}_{2},
3\overline{\varepsilon}_{1}+2\overline{\varepsilon}_{2}\}$. See
Section \ref{sect-2}.3 for detail.

\medskip
We consider the map
$$
h:\quad \mathbb{Z}\widehat{I}\longrightarrow\mathbb{Z}\mathcal{I},
\qquad h(\alpha)_{i}=\sum_{\rho\in{\rm irr} G_{i}}\alpha_{i\rho}
$$
for any $\alpha=\sum_{(i,
\rho)\in\widehat{I}}\alpha_{i\rho}\varepsilon_{(i\rho)
\in\widehat{I}}\in\mathbb{Z}\widehat{I}$. The restriction of
$h: \Delta_{\widehat{Q}}\rightarrow\Delta_{\Gamma}$ is a surjective,
this means that for any positive root $\beta$ of $\Gamma$, there
exists an indecomposable $\widehat{Q}$-representation $X$ such that
$h({\bf dim} X)=\beta$. For example, we consider the positive root
$\overline{\varepsilon}_{1}+\overline{\varepsilon}_{2}\in\Delta_{\Gamma}$.
Then, we have the following indecomposable
$\widehat{Q}$-representation $X_{(\rho_{3}\sigma_{0})}$:
$$~~~~~~~~~~~~~~~~~~~~~~~\setlength{\unitlength}{1mm}
\begin{picture}(50,17)
\put(-3,7){$\mathbbm{k}$}               \put(12,0){$0$}
\put(12,14){$\mathbbm{k}$}              \put(-20,7){$0$}
\put(11,2){\vector(-2,1){10}}     \put(11,14){\vector (-2,-1){10}}
\put(-17,8){\vector (1,0){13}}          \put(-11,9){\small$0$}
\put(2,12){\small$1$}                 \put(2,2){\small$0$}
\put(47,7){\small$0$}                   \put(60,0){$0$}
\put(60,14){$0$}                          \put(30,7){$0$}
\put(59,2){\vector (-2,1){10}}     \put(59,14){\vector (-2,-1){10}}
\put(33,8){\vector (1,0){13}}           \put(39,9){\small$0$}
\put(52,12){\small$0$}                  \put(52,2){\small$0$}
\end{picture}
$$
and obviously, $h({\bf dim}X_{(\rho_{3}\sigma_{0})})
=\overline{\varepsilon}_{1}+\overline{\varepsilon}_{2}$.
Furthermore, for any $0\leq l\leq 5$, $0\leq j\leq 1$ and
$l\not\equiv j\mod{2}$, we define the $\widehat{Q}$-representation
$X_{(\rho_{l}\sigma_{j})}=(X_{i\rho}, X_{\alpha})$ by
$$
X_{i\rho}=\left\{\begin{array}{ll}\mathbbm{k}, & \mbox{ if }
(i, \rho)=(1, \rho_{l}) \mbox{ or } (2, \sigma_{j});\\
0 , & \mbox{ otherwise. }\end{array}\right.\qquad
X_{\alpha}=\left\{\begin{array}{ll}1, & \mbox{ if }
\alpha=\alpha_{l};\\
0 , & \mbox{ otherwise. }
\end{array}\right.
$$
Then, it is easy to see that the set of all indecomposable
$\widehat{Q}$-representations with $h({\bf dim}X
)=\overline{\varepsilon}_{1}+\overline{\varepsilon}_{2}$ is the set
$$
\left\{X_{(\rho_{l}\sigma_{j})}\mid 0\leq l\leq 5, ~0\leq j\leq 1 \mbox{
and } l\not\equiv j\mod{2}\right\},
$$
and which is just the orbit of $X_{(\rho_{1}\sigma_{0})}$ under that
action of $G$. Similarly, for any positive real root
$\beta=h(\alpha)\in\Delta_{\Gamma}$, there are $|H_{\alpha}|$ (up to
isomorphism) indecomposable $\widehat{Q}$-representations $X$ such
that $h({\bf dim}X)=\beta.$
\end{example}

\begin{example} \label{exa5-2}
Let $Q=(I, E)$ be the quiver
$$~~~~~~~~~~~~~~~~~~~~~~~\setlength{\unitlength}{1mm}
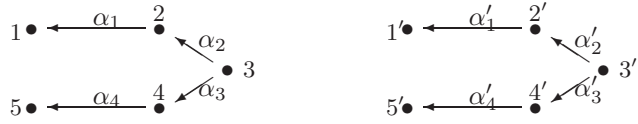
\begin{picture}(50,15)
\put(10,5){$\bullet$}              \put(1,0){$\bullet$}
\put(1,10.5){$\bullet$}           \put(-16,10.5){$\bullet$}
\put(-16,0){$\bullet$}            \put(-18,0){\small $5$}
\put(13,5){\small $3$}              \put(1,2){\small $4$}
\put(-18,10){\small $1$}             \put(1,12.5){\small $2$}
\put(9,7){\vector (-3,2){5}}       \put(9,5){\vector (-3,-2){5}}
\put(0,11.5){\vector (-1,0){13}}  \put(0,1){\vector (-1,0){13}}
\put(-7,12){$\alpha_{1}$}          \put(-7,1.5){$\alpha_{4}$}
\put(7,9){$\alpha_{2}$}            \put(7,2.5){$\alpha_{3}$}
\put(60,5){$\bullet$}              \put(51,0){$\bullet$}
\put(51,10.5){$\bullet$}           \put(34,10.5){$\bullet$}
\put(34,0){$\bullet$}            \put(32,0){\small $5'$}
\put(63,5){\small $3'$}              \put(51,2){\small $4'$}
\put(32,10){\small $1'$}             \put(51,12.5){\small $2'$}
\put(59,7){\vector (-3,2){5}}       \put(59,5){\vector (-3,-2){5}}
\put(50,11.5){\vector (-1,0){13}}  \put(50,1){\vector (-1,0){13}}
\put(43,12){$\alpha'_{1}$}          \put(43,1.5){$\alpha'_{4}$}
\put(57,9){$\alpha'_{2}$}            \put(57,2.5){$\alpha'_{3}$}
\end{picture}
$$
and $G=\langle a\rangle\times\langle
b\rangle\cong\mathbb{Z}/2\mathbb{Z}\times\mathbb{Z}/2\mathbb{Z}$.
The action of $G$ on $\mathbbm{k}Q$ is given as follows
\begin{center}
\begin{tabular*}{9cm}{@{\extracolsep{\fill}}l|lllllllllllr}
& $e_{1}$ & $e_{2}$ & $e_{3}$ & $e_{4}$ & $e_{5}$ & $e_{1'}$ &
$e_{2'}$ & $e_{3'}$ & $e_{4'}$ & $e_{5'}$  \\
\hline $a$ & $e_{5}$ & $e_{4}$ & $e_{3}$ & $e_{2}$ & $e_{1}$ &
$e_{5'}$ & $e_{4'}$ & $e_{3'}$ & $e_{2'}$ & $e_{1'}$  \\
$b$ & $e_{1'}$ & $e_{2'}$ & $e_{3'}$ & $e_{4'}$ & $e_{5'}$
& $e_{1}$ & $e_{2}$ & $e_{3}$ & $e_{4}$ & $e_{5}$ \\
\end{tabular*}

\medskip
\begin{tabular*}{8cm}{@{\extracolsep{\fill}}l|llllllllllllllr}
& $\alpha_{1}$ & $\alpha_{2}$ & $\alpha_{3}$ & $\alpha_{4}$ &
$\alpha'_{1}$ & $\alpha'_{2}$ & $\alpha'_{3}$ & $\alpha'_{4}$\\
\hline $a$ & $\alpha_{4}$ & $\alpha_{3}$ & $\alpha_{2}$ &
$\alpha_{1}$ & $\alpha'_{4}$ & $\alpha'_{3}$ & $\alpha'_{2}$ &
$\alpha'_{1}$ \\
$b$ & $\alpha'_{1}$ & $\alpha'_{2}$ & $\alpha'_{3}$ & $\alpha'_{4}$
& $\alpha_{1}$ & $\alpha_{2}$ & $\alpha_{3}$ & $\alpha_{4}$ \\
\end{tabular*}\end{center}
where $e_{i}$ is the idempotent element of $\mathbbm{k}Q$
corresponding to the vertex $i$. Taken $\mathcal{I}=\{1, 2, 3\}$,
then the generalized McKay quiver of $(Q, G)$ is
$$~~~~~~~~~~~~~~~~~~~~~~~\setlength{\unitlength}{1mm}
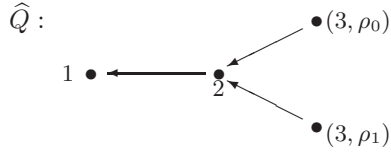
\begin{picture}(50,18)
\put(25,8){$\bullet$}                 \put(38,1){$\bullet$}
\put(38,15){$\bullet$}               \put(8,8){$\bullet$}
\put(40,15){\small $(3,\rho_{0})$}     \put(40,0){\small
$(3,\rho_{1})$}                        \put(25,6){\small $2$}
\put(5,8){\small $1$}          \put(37,15){\vector (-2,-1){10}}
\put(37,3){\vector (-2,1){10}}       \put(24,9){\vector (-1,0){13}}
\put(-2,15){$\widehat{Q} :$}
\end{picture}
$$
where $\rho_{0}$, $\rho_{1}$ are the non-isomorphism irreducible
representations of $G_{3}=\langle
a\rangle\cong\mathbb{Z}/2\mathbb{Z}$. Reindexing the vertex set
$\widehat{I}=\{1, 2, (3,\rho_{0}), (3,\rho_{1})\}$ by $\{1, 2, 3,
4\},$ the Cartan matrix of $\widehat{Q}$ is
$$
A=(a_{ij})={\small\left(
\begin{array}{cccc}2&-1&0&0\\-1&2&-1&-1\\0&-1&2&0\\0&-1&0&2\\
\end{array}\right)}.
$$
The Lie algebra $\mathfrak{g}:=\mathfrak{g}(\widehat{Q})$ is
generated by $\{x_{i}, y_{i}, h_{i}\mid 1\leq i\leq4\}$ satisfying
the relations
$$
\begin{array}{llll} \quad[h_{i}, h_{j}]=0,
\qquad\qquad\qquad & [x_{i}, y_{j}]=\delta_{ij}h_{i};\\
\quad[h_{i}, x_{j}]=a_{ij}x_{j}, &[h_{i}, y_{j}]=-a_{ij}y_{j};\\
\quad(\ad x_{i})^{1-a_{ij}}(x_{j})=0, & (\ad
y_{i})^{1-a_{ij}}(y_{j})=0, \qquad i\neq j.
\end{array}
$$
In this case, the valued graph $\Gamma$ of $(Q, G)$ is
$$~~~~~~~~~~~~~~~~~~~~~~~\setlength{\unitlength}{1mm}
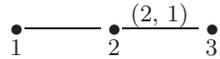
\begin{picture}(6,7)
\put(0,2.5){$\bullet$}                \put(13,2.5){$\bullet$}
\put(2,3.5){\line (1,0){10}}          \put(-13,2.5){$\bullet$}
\put(-11,3.5){\line (1,0){10}}          \put(-13,0){\small $1$}
\put(0,0){\small $2$}             \put(13,0){\small $3$}
\put(3,4.5){\small $(2,\,1)$}
\end{picture}
$$
with the Cartan Matrix
$$
C={\small\left(
\begin{array}{ccc}2&-1&0\\-1&2&-1\\0&-2&2\\\end{array}\right)}.
$$
The Lie algebra $\mathfrak{g}(\Gamma)$ is generated by $\{X_{i},
Y_{i}, H_{i}\mid 1\leq i\leq3\}$ satisfying the relations
\begin{eqnarray} \label{seqn}
\begin{array}{llll} \quad[H_{i},
H_{j}]=0, \qquad\qquad\qquad & [X_{i}, Y_{j}]=\delta_{ij}H_{i};\\
\quad[H_{i}, X_{j}]=c_{ij}X_{j}, & [H_{i}, Y_{j}]=-c_{ij}Y_{j};\\
\quad(\ad X_{i})^{1-c_{ij}}(X_{j})=0, & (\ad
Y_{i})^{1-c_{ij}}(Y_{j})=0, \qquad i\neq j.
\end{array}
\end{eqnarray}
As discussed in Section \ref{sect-3}.2, we see that  the vertices
$(3, \rho_{0})$ and $(3, \rho_{1})$ of $\widehat{Q}$ are in the same
$G$-orbit. Therefore, the Lie algebra $\mathfrak{g}^{\overline{G}}$
is generated by
$$
\{\overline{x}_{i},
\overline{y}_{i}, \overline{h}_{i}\mid 1\leq i\leq3\},
$$
where $\overline{x}_{i}=x_{i}$, $\overline{y}_{i}=y_{i}$,
$\overline{h}_{i}=h_{i}$ for $i=1,2$, and
$\overline{x}_{3}=x_{3}+x_{4}$, $\overline{y}_{3}=y_{3}+y_{4}$,
$\overline{h}_{3}=h_{3}+h_{4}$, satisfying the relations
(\ref{seqn}). Then, it is easy to see that the map
$$
\Phi : \quad
\mathfrak{g}(\Gamma)\longrightarrow\mathfrak{g}^{\overline{G}}
$$
given by
$$
\Phi(X_{i})=\overline{x}_{i}, \quad \Phi(Y_{i})=\overline{y}_{i},
\quad \Phi(H_{i})=\overline{h}_{i}
$$
is an Lie algebra isomorphism.

\medskip
At last, we consider quivers of $A$-type and $D$-type,
$$~~~~~~~~~~~~~~~~~~~~~~~\setlength{\unitlength}{1mm}
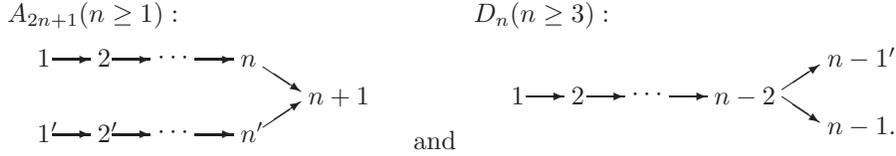
\begin{picture}(45,20)
\put(-38,17){$A_{2n+1} (n\geq 1):$}     \put(-34,11){$1$}
\put(-34,1){$1'$}                      \put(-32,2){\vector (1,0){5}}
\put(-32,12){\vector (1,0){5}}         \put(-26,11){$2$}
\put(-26,1){$2'$}                     \put(-24,2){\vector (1,0){5}}
\put(-24,12){\vector (1,0){5}}        \put(-18,1.5){$\cdots$}
\put(-18,11.5){$\cdots$}              \put(-13,2){\vector (1,0){5}}
\put(-13,12){\vector (1,0){5}}        \put(-7,11){$n$}
\put(-7,1){$n'$}                     \put(-4,3){\vector (3,2){5}}
\put(-4,11){\vector (3,-2){5}}        \put(2,6){$n+1$}
\put(24,17){$D_{n} (n\geq 3):$}         \put(56,6){$n-2$}
\put(65,8){\vector (3,2){5}}            \put(71,11){$n-1'$}
\put(65,7){\vector (3,-2){5}}           \put(71,2){$n-1$.}
\put(29,6){$1$}                         \put(31,7){\vector (1,0){5}}
\put(37,6){$2$}                         \put(39,7){\vector (1,0){5}}
\put(45,6.5){$\cdots$}                 \put(50,7){\vector (1,0){5}}
\put(16,0){$\mbox{and}$}
\end{picture}$$
They have the same quiver isomorphism group
$G=\mathbb{Z}/2\mathbb{Z}$. In these cases, we have
$$
\begin{tabular}{|c|c|c|c|c|c|}
\hline
$\quad\begin{picture}(8,14) \put(0,1){$Q$}\end{picture}\quad$ &
$\quad\begin{picture}(8,14) \put(0,1){$G$}\end{picture}\quad$ &
$\quad\begin{picture}(8,14) \put(0,1){$\Gamma$}\end{picture}\quad$ &
$\quad\begin{picture}(8,14)
\put(0,1){$\widehat{Q}$}\end{picture}\quad$ &
$\quad\begin{picture}(8,14)
\put(0,1){$\widehat{\Gamma}$}\end{picture}\quad$ &
$\quad\begin{picture}(100,14) \put(25,1){$\mbox{ Conclusion
}$}\end{picture} \quad$ \\ \hline $\quad\begin{picture}(8,16)
\put(-8,2){$A_{2n+1}$}\end{picture}\quad$ &
$\quad\begin{picture}(8,16)
\put(-6,2){$\mathbb{Z}/2\mathbb{Z}$}\end{picture}\quad$ &
$\quad\begin{picture}(8,16) \put(-6,2){$C_{n+1}$}\end{picture}\quad$
& $\quad\begin{picture}(8,16)
\put(-6,2){$D_{n+2}$}\end{picture}\quad$ &
$\quad\begin{picture}(8,16) \put(-6,2){$B_{n+1}$}\end{picture}\quad$
& ${\begin{picture}(60,16) \put(-20,2){$\mathfrak{g}(C_{n+1})\cong
\mathfrak{g}(D_{n+2})^{\mathbb{Z}/2\mathbb{Z}}$}
\end{picture}}$\\ \hline
$\quad\begin{picture}(8,16) \put(0,2){$D_{n}$}\end{picture}\quad$ &
$\quad\begin{picture}(8,16)
\put(-6,2){$\mathbb{Z}/2\mathbb{Z}$}\end{picture}\quad$ &
$\quad\begin{picture}(8,16) \put(-6,2){$B_{n-1}$}\end{picture}\quad$
& $\quad\begin{picture}(8,16)
\put(-8,2){$A_{2n-1}$}\end{picture}\quad$ &
$\quad\begin{picture}(8,16) \put(-6,2){$C_{n-1}$}\end{picture}\quad$
& ${\begin{picture}(60,16) \put(-20,2){$\mathfrak{g}(B_{n-1})\cong
\mathfrak{g}(A_{2n-1})^{\mathbb{Z}/2\mathbb{Z}}$}
\end{picture}}$\\
\hline
\end{tabular}
$$
where $C_{n}$ and $B_{n}$ is the $C$-type and $B$-type Dynkin
diagram, respectively.
\end{example}

\bibliographystyle{amsplain}

\end{document}